\documentclass[12pt,reqno]{amsproc}

\usepackage{latexsym} 
  \usepackage[all]{xy}
  \usepackage{amsfonts} 
  \usepackage{amsthm} 
  \usepackage{amsmath} 
  \usepackage{amssymb}
  \usepackage{pifont}  
  \usepackage{enumerate}
\newtheorem{thm}{Theorem}[section]
 \newtheorem{cor}[thm]{Corollary}
 \newtheorem{lem}[thm]{Lemma}
 \newtheorem{prop}[thm]{Proposition}
 
 \theoremstyle{definition}
 \newtheorem{defn}[thm]{Definition}

  \theoremstyle{remark}
 \newtheorem{rem}[thm]{Remark}
 \newtheorem{ex}[thm]{Example}
  \newtheorem{conv}[thm]{Conventions}
 \numberwithin{equation}{section}
 
 \def\bs{\subsection}%{\begin{statement}}
\def\es{\smallskip}%{\end{statement}}
  \swapnumbers
    \def\eps{\varepsilon}

  \newcounter{zlist} 
  \newenvironment{zlist}{\begin{list}{(\arabic{zlist})}{ 
  \usecounter{zlist}\leftmargin2.5em\labelwidth2em\labelsep0.5em 
  \topsep0.6ex%\itemsep0.3ex plus0.2ex minus0.3ex 
  \parsep0.3ex plus0.2ex minus0.1ex}}{\end{list}}

  \newcounter{blist} 
  \newenvironment{blist}{\begin{list}{(\alph{blist})}{ 
  \usecounter{blist}\leftmargin2.5em\labelwidth2em\labelsep0.5em 
  \topsep0.6ex %\itemsep0.3ex plus0.2ex minus0.3ex 
  \parsep0.3ex plus0.2ex minus0.1ex}}{\end{list}} 

  \newcounter{rlist} 
  \newenvironment{rlist}{\begin{list}{(\roman{rlist})}{ 
  \usecounter{rlist}\leftmargin2.5em\labelwidth2em\labelsep0.5em 
  \topsep0.6ex %\itemsep0.3ex plus0.2ex minus0.3ex 
  \parsep0.3ex plus0.2ex minus0.1ex}}{\end{list}}

\newcommand{\Ker}{{\rm Ker}}

\newcommand{\im}{{\rm Im}\,}

\def\NN{{\mathbb N}}

\def\ZZ{{\mathbb Z}}

\newcommand{\Cc}{\mathcal{C}}

\newcommand{\Ii}{\mathcal{I}}

\def\p{\mathbf{p}}

\def\*C{{}^*\hspace*{-1pt}{\Cc}}

\def\text#1{{\rm {\rm #1}}}

\def\mod{\mathbf{Mod}}

 \def\1{\mathbf{1}}

 \def\tr{\mathrm{tr} }

  \def\la{\triangleright}
 
   \def\ra{\triangleleft}
\def\id{\mathrm{id}}
\def\hrd {\mathbf{Heap}}
\def\ahrd{\mathbf{Ah}}

\def\di{\diamond}
\def\truss {\mathbf{Trs}}
\def\lto{\longmapsto}
\def\lra{\longrightarrow}
\def\lhom#1#2#3{\mathrm{Hom}_#1(#2,#3)}
\def\rhom#1#2#3{\mathrm{Hom}_{-,#1}(#2,#3)}
\def\lrhom#1#2#3#4{\mathrm{Hom}_{#1,#2}(#3,#4)}

\begin{document}

\title{Trusses: Paragons, ideals and modules}

\author{Tomasz Brzezi\'nski}

\address{
Department of Mathematics, Swansea University, 
Swansea University Bay Campus,
Fabian Way,
Swansea,
  Swansea SA1 8EN, U.K.\ \newline \indent
Department of Mathematics, University of Bia{\l}ystok, K.\ Cio{\l}kowskiego  1M,
15-245 Bia\-{\l}ys\-tok, Poland}

\email{T.Brzezinski@swansea.ac.uk}



\date{January 2019}

\begin{abstract}
Trusses, defined as sets with a suitable ternary and a binary operations, connected by the distributive laws, are studied from a ring and module theory point of view. The notions of ideals and paragons in trusses are introduced and several constructions of trusses are presented. A full classification of truss structures on the Abelian group of integers is given. Modules over trusses are defined and  their basic properties and examples are analysed. In particular, the sufficient and necessary condition for a sub-heap of a module to induce a module structure on the quotient heap is established.
\end{abstract}

\maketitle

\setcounter{tocdepth}{2}
\tableofcontents

\section{Introduction}
A heap\footnote{Pr\"ufer and Baer use the word  {\em die Schar}, which translates as {\em herd} or {\em flock}. The term  {\em heap}  from Russian {\em gruda}, a word-play with {\em grupa} meaning a group, and according to \cite[p.\ 96]{BerHau:cog} introduced in \cite{Sus:the}, seems to be more widespread in the English literature and we use it here as do Hollings and Lawson in recently published English collection of the key works of V.V.\ Wagner \cite{HolLaw:Wag}. The term {\em torsor} is also used.}, a notion introduced by H.\ Pr\"ufer \cite{Pru:the} and R.\ Baer \cite{Bae:ein}, is an algebraic system consisting of a set and a ternary operation satisfying simple conditions (equivalent to conditions satisfied in a group by operation $(x,y,z)\lto xy^{-1}z$),  which can be understood as a group in which the neutral element has not been specified. A choice of any element in a heap can reduce the ternary operation to a binary operation that makes the underlying set into a group in which the chosen element is the neutral element. Enriching a heap with additional (associative) binary operation which distributes over the ternary heap operation seems to be a natural progression that mimics the process which leads from groups to rings. This has been attempted in \cite{Brz:tru}, resulting in the introduction of  the notion of a {\em truss}, and has triggered irresistible (at least to the writer of these words) mathematical  curiosity as to the nature of such a simple system, its structure and representations.

Na\"\i vely, a truss can be understood as  a ring in which the Abelian group of addition has no specified neutral element. A choice of an element makes the underlying heap operation into a binary Abelian group operation (with the chosen element being the zero). By making an arbitrary choice, however, one does not necessarily obtain the usual distribution of multiplication over the addition, but a more general distributive law. Apart from the usual ring-theoretic distributive law (a chosen element needs to have a particular absorption property, see Definition~\ref{def.braceable}), making a suitable choice one obtains the distributive law which has recently been made prominent in the theory of solutions of the set-theoretic Yang-Baxter equation \cite{Dri:uns} and radical rings through the introduction of {\em braces} by W.\ Rump in \cite{Rum:bra} (see also \cite{CedJes:bra}).

A {\em left brace} is a set $A$ together with two group operations $\cdot$  and commutative $+$, which satisfy the following {\em brace distributive law}, for all $a,b,c\in A$,
\begin{equation}\label{brace.law}
a\cdot (b+c) = a\cdot b  - a +a\cdot c.
\end{equation}
Similarly a {\rm right brace} is defined and the system which is both left and right brace (with the same operations) is called a {\em two-sided brace}. If the commutativity assumption on $+$ is dropped, the resulting systems are qualified by an adjective {\em skew} (as in left skew brace, right skew brace, two-sided skew brace) \cite{GuaVen:ske}. It is a matter of simple calculation to check that the brace distributive law forces both group structures to share the same neutral element. An equally simple calculation confirms that operation $\cdot$ distributes over the ternary heap operation $(a,b,c)\lto a -b +c$.

The aim of this paper is to initiate systematic studies of trusses using the same approach as in ring theory. We begin in Section~\ref{sec.heap} with a review of basic properties of heaps. This section does not pretend to any originality, its purpose being a repository of facts about heaps that are used later on.  At the start of Section~\ref{sec.truss} we recall the definition of a truss from \cite{Brz:tru}. Although this definition can be given in a number of equivalent ways, in this text we concentrate on the one which characterises a truss as a heap together with an associative multiplication distributing over the ternary heap operation. This definition is closest to the prevailing heap philosophy of working without specifying elements of particular nature. If a truss contains an element which has an absorption property (in the sense that multiplication by this element always gives back this element), then the multiplication distributes over addition induced from the heap operation by this element (in the usual ring-theoretic sense); thus this specification gives a ring. If a multiplicative semigroup is a monoid, then the multiplication distributes over the addition induced from the heap operation by the identity according to the brace distributive law \eqref{brace.law}. Note that if a truss contains both an absorbing element and identity for the multiplication there is a freedom of choice of the element specifying addition; traditionally one chooses the absorber as the zero for the addition (and obtains a ring)  rather than the identity which would result in a brace-type algebraic system.

Next we describe actions of the multiplicative semigroup of a truss induced by the distributive law. Subsequently, these actions  play a key role in the definition of {\em paragons}: a paragon is a sub-heap that is closed under these actions;  it is the closeness under the actions not  under the semi-group multiplication that characterises sub-heaps of a truss such that the quotient  heap is a truss. Ideals, defined following the ring-theoretic intuition, are examples of paragons. We conclude Section~\ref{sec.truss} with a range of examples. First, we show that any Abelian heap can be made into a truss in (at least) three different ways. Then we prove that the set of endomorphisms of an Abelian heap is a truss with respect to the pointwise heap operation and composition of morphisms. This truss is particularly important for the definition of modules. We connect further the endomorphism truss with a semi-direct product of a heap with an endomorphism monoid of any associated Abelian group. This allows one for explicit construction of examples. Finally we list all truss structures on the heap of integers (with the heap operation induced by the addition of numbers). Apart from two non-commutative truss structures that can be defined on any Abelian heap, all other truss multiplications on $\ZZ$ are commutative and in bijective correspondence  with nontrivial idempotents in the ring of two-by-two integral matrices. Up to isomorphisms commutative truss structures on $\ZZ$ are in one-to-one correspondence with orbits of the action (by conjugation) of the infinite dihedral group $D_\infty$, realised as a particular subgroup of $GL_2(\ZZ)$, on this set of idempotents.

Section~\ref{sec.mod} is devoted to the introduction and description of basic properties of modules over trusses. Since the endomorphism monoid of any Abelian heap is a truss, one can study truss homomorphisms with the truss as a domain and  the endomorphism truss of an Abelian heap as a codomain. In the same way as modules over rings, heaps together with truss homomorphisms to their endomorphism trusses are understood as (left) modules over a truss. Equivalently, modules can be characterised as heaps with an associative and distributive (over the ternary heap operations) action of a given truss. We study examples of modules, in particular modules over a ring of integers understood as a truss, and give basic constructions such as products of modules or module structures on sets of functions with a module as a codomain. Similarly to modules over rings, homomorphisms of modules over trusses can be equipped with actions and thus turned into modules. We show that both paragons and ideals are modules, and then study submodules and quotients. Modules obtained as quotients by submodules have a particular absorption property that allows one to convert truss-type distributive law into a ring-type distributive law for actions. In contrast to ring theory and in complete parallel to the case of trusses and their paragons, a more general quotient procedure is possible. In a similar way to trusses, whereby with any element of a truss one can associate an action of the multiplicative semigroup, the choice of an element of a module yields an induced action of a truss on this module. It turns out that the kernel of a module homomorphism is a sub-heap closed under this induced action. The quotient of a module by any sub-heap closed under this induced action has an induced module structure.

\section{Heaps}\label{sec.heap}
\bs{Heaps: definition}\label{heap}
Following \cite[page 170]{Pru:the},  \cite[page 202, footnote]{Bae:ein}  or \cite[Definition~2]{Cer:ter},
a  {\em heap} is a pair $(H,[---])$ consisting of a non-empty set $H$ and a ternary operation
$$
[---] : H\times H\times H\to H, \qquad (x,y,z)\mapsto [x,y,z],
$$
satisfying the following conditions, for all $v,w,x,y,z \in H$,
\begin{equation}\label{heap.axioms}
\begin{aligned}
& \mbox{\bf associativity:}\qquad  [v,w,[x,y,z]] = [[v,w,x,],y,z],\\
& \mbox{\bf Mal'cev identities:}\qquad\quad [x,x,y] = y = [y,x,x].
\end{aligned}
\end{equation}

A heap $(H,[---])$ is said to be {\em Abelian}, if, for all $x,y,z$,
\begin{equation}\label{Abelian}
[x,y,z] = [z,y,x].
\end{equation}

A {\em heap morphism} from $(H, [---])$ to $(\tilde{H},   [---])$ is a function $\varphi: H\to \tilde{H}$ respecting the ternary operations, i.e., such that for all $x,y,z$,
\begin{equation}\label{morph}
\varphi([x,y,z]) = [\varphi(x),\varphi(y),\varphi(z)].
\end{equation}
The category of heaps is denoted by $\hrd$ and the category of Abelian heaps is denoted by $\ahrd$.

For any $n\in \NN$ we also introduce the operations
\begin{equation}\label{multi}
\begin{aligned}
{}[[-\ldots-]]_n &: H^{2n+1} \lra H, \\
 [[x_1, x_2, \ldots, x_{2n+1}]]_n
 & = [[\ldots[ [x_1,x_2,x_3],x_4,x_5],\ldots], x_{2n},x_{2n+1}].
\end{aligned}
\end{equation}
In view of the associativity of the heap operation, various placements of $[---]$ (all moves by two places in general or any move in the case of an Abelian heap) lead to the same outcome. Mal'cev identities imply that any symbol appearing twice in consecutive places in $[[-\ldots-]]$ can be removed, and that $[[-\ldots-]]$ is an idempotent operation.

There is an obvious forgetful functor from the category of heaps to the category of sets. Any singleton set $\{*\}$ has a trivial heap operation $[***] = *$ (the only function with $\{*\}$ as a codomain). We refer to $(\{*\}, [---])$ as to a {\em trivial heap}. The trivial heap is the terminal but not initial object in $\hrd$\footnote{Since the definition of a heap involves only universal quantifiers, the notion can be extended to include the empty set, which becomes the initial object of thus extended category of heaps.}. Any function $\{*\}\to H$ is a heap homomorphism since  $[---]$ is an idempotent operation. In particular, global points of heaps coincide with  the points of their underlying sets.

\es

\bs {Heaps and groups}
Heaps correspond to groups in a way similar to that in which affine spaces correspond to vector spaces: heaps can be understood as  groups without a specified identity element; fixing an identity element converts a heap into a group. 

\begin{lem}\label{lem.heap.group}
\begin{zlist} 
\item Given a group $(G,\diamond, 1_\di)$, let
\begin{equation}\label{group.to.heap}
[---]_\diamond : G\times G\times G\longrightarrow G, \qquad [x,y,z]_\diamond = x\diamond y^{-1}amond \diamond z.
\end{equation}
Then $(G, [---]_\diamond)$ is a heap. Furthermore, any homomorphism of groups is a homomorphism of corresponding heaps.
\item Given a heap  $(H, [---])$  and $e\in H$, let
\begin{equation}\label{heap.to.group}
-\di_e- : H\times H\to H, \qquad x\di_e y = [x, e, y].
\end{equation}
Then $(H,\di_e, e)$ is a group, known as a {\em retract of  $H$}. Furthermore, if $\varphi$ is a morphism of heaps from $(H,[---])$ to $(\tilde{H}, [---])$ then for all $e\in H$ and $\tilde{e}\in \tilde{H}$, the functions
\begin{subequations}
\begin{equation}\label{mor.to.mor}
\widehat{\varphi}: H \lra \tilde{H},  \qquad x\lto [\varphi(x),\varphi(e),\tilde{e}],
\end{equation}
\begin{equation}\label{mor.to.mor.2}
\widehat{\varphi}^\circ: H \lra \tilde{H},  \qquad x\lto [\tilde{e},\varphi(e),\varphi(x)],
\end{equation}
\end{subequations}
are homomorphism of groups from $(H,\di_e, e)$ to $(\tilde{H},\di_{\tilde{e}}, \tilde{e})$.
\item Let $(H, [- - -])$ be a heap. Then for all $e,f \in H$,
\begin{blist}
\item Groups $(H,\di_e,e)$ and $(H,\di_f, f)$ are mutually isomorphic.
\item $[- - -]_{\di_e} = [- - -]$.
\end{blist}
\end{zlist}
\end{lem}
This lemma, whose origins go back to Baer \cite{Bae:ein}, can be proven by direct checking of group or heap axioms. We only note in passing that the inverse in $(H,\di_e,e)$ is given by
\begin{equation}\label{inv.e}
x^{-1} = [e,x,e],
\end{equation}
while the isomorphism from $(H,\di_e,e)$ to  $(H,\di_f, f)$ is given by 
\begin{equation}\label{iso.ef}
\tau_e^f: H\lra H, \quad x\lto x\di_e f = [x,e,f].
\end{equation}
The group and hence also the heap automorphism $\tau_e^f$, whose inverse is $\tau_f^e$, will be frequently used, and we refer to it as a {\em neutral element swap} or simply as a {\em swap automorphism}. The correspondence of Lemma~\ref{lem.heap.group}, which can be understood as an isomorphism between the category of groups and {\em based heaps}, i.e.\ heaps with a distinguished element and morphisms that preserve both the heap operations and distinguished elements, extends to Abelian groups and heaps. 

\begin{rem}\label{rem.cat.heap}
In view of the preceding discussion  the category of based heaps is the same as the co-slice category $(\{*\} \!\downarrow\! \hrd)$ consisting of morphisms in $\hrd$ with the domain $\{*\}$ and with morphisms given by  commutative triangles in $\hrd$,
$$
\xymatrix{H \ar[rr] &&\tilde{H} \\& \{*\}\ar[ul]\ar[ur] &}.
$$
Lemma~\ref{lem.heap.group} establishes an isomorphism of $(\{*\} \!\downarrow\! \hrd)$ with the category of groups, while formula \eqref{mor.to.mor} gives a way of converting any morphism in $\hrd$ into a morphism in $(\{*\} \!\downarrow\! \hrd)$ that is compatible with composition.
\end{rem}

The equality of heap operations in Lemma~\ref{lem.heap.group}(3)(a) allows for not necessarily desired from the philosophical viewpoint, but technically convenient usage of group theory in study of heaps. Starting with a heap, one can make a choice of an element, thus converting a heap into a group, and performing all operations using the resulting binary operation. At the end the result can be converted back to the heap form.  As an example of this procedure, one can prove the following

\begin{lem}\label{lem.equal}
Let $(H,[- - -]$)  be a heap. 
\begin{zlist}
\item If $e,x,y\in H$ are such that $[x,y,e]=e$ or $[e,x,y] = e$, then $x=y$. 
\item For all $v,w,x,y,z\in H$
\begin{equation}\label{assoc+}
[v,w,[x,y,z]] = [v,[y,x,w], z].
\end{equation}
\item For all $x,y,z\in H$,
$$
[x,y,[y,x,z]] = [[z,x,y],y,x] = [x,[y,z,x],y] =z.
$$
In particular, in the expression $[x,y,z]=w$ any three elements determine the fourth one.
\item If $H$ is Abelian, then, for all $x_i,y_i,z_i\in H$, $i=1,2,3$,
$$
\left[\left[x_1,x_2,x_3\right],\left[y_1,y_2,y_3\right], \left[z_1,z_2,z_3\right]\right] = \left[\left[x_1,y_1,z_1\right],\left[x_2,y_2,z_2\right], \left[x_3,y_3,z_3\right]\right].
$$
\end{zlist}
\end{lem}
\begin{proof}
(1) Since
$
e= [x,y,e] = x\di_e y^{-1}\di_e e = x\di_e y^{-1},
$
we immediately obtain that $x=y$ as required (and similarly for the second statement in assertion (1)).

Assertion (2) is proven by an equally simple exercise, while (3) follows by associativity, Mal'cev identities and (2).

(4) Take any $e\in H$. Using the commutativity of the induced operation $\di_e$, we can compute
$$
\begin{aligned}
\left[\left[x_1,x_2,x_3\right],\left[y_1,y_2,y_3\right], \left[z_1,z_2,z_3\right]\right] & = x_1\di_e x_2^{-1} \di_e x_3 \di_e \left(y_1\di_ey_2^{-1}\di_ey_3\right)^{-1}\di_ez_1\di_e z_2^{-1}\di_ez_3\\
&= x_1\di_e y_1^{-1} \di_e z_1 \di_e x_2^{-1}\di_ey_2\di_ez_2^{-1}\di_ex_3\di_e y_3^{-1}\di_ez_3\\
&=\left[\left[x_1,y_1,z_1\right],\left[x_2,y_2,z_2\right], \left[x_3,y_3,z_3\right]\right],
\end{aligned}
$$
as required.
\end{proof}

\begin{conv}\label{conv.Abelian}
We will use the additive notation for group structures associated to an Abelian heap $(H,[---])$. Thus, for any $e\in H$, 
\begin{subequations}
\begin{equation}\label{ab.+}
x +_e y := [x,e,y], \qquad \mbox{for all $x,y\in H$},
\end{equation}
\begin{equation}\label{ab.-}
-_e x := [e,x,e], \qquad \mbox{for all $x\in H$},
\end{equation}
\begin{equation}\label{ab.sum}
\sum_{i=1}^n{}\!^e\, x_i  := x_1+_ex_2+_e\ldots +_ex_n, \qquad \mbox{for all $x_1,\ldots,x_n\in H$},
\end{equation}
\end{subequations}
\end{conv}
\es

\bs{Sub-heaps}\label{sub-heap} 
In this section we look at sub-heaps and normal sub-heaps.
\begin{defn}\label{def.subheap}
Let $(H, [- - -])$ be a heap.  A subset $S\subseteq H$ is a {\em sub-heap}, if it is closed under $[- - -]$, i.e., for all $x,y,z\in S$, $[x,y,z]\in S$. A sub-heap $S$ of $(H, [- - -])$ is said to be {\em normal} if there exists $e\in S$ such that, for all $x \in H$ and $s\in S$ there exists $t\in S$ such that
\begin{equation}\label{normal}
[x,e,s] =[t,e,x].
\end{equation}
\end{defn}
Axioms of a heap allow for an unbridled interplay between existential and universal quantifiers.
\begin{lem}\label{lem.normal}
A sub-heap $S$ of $(H, [- - -])$ is  normal if and only if,  for all $x \in H$ and $e,s\in S$ there exists $t\in S$ such that \eqref{normal} holds or, equivalently, $[[x,e,s],x,e]\in S$.
\end{lem}
\begin{proof}
Clearly, if $t$ exists for all $x,e,s$, then $S$ is normal. Conversely, suppose that \eqref{normal} holds for a fixed $e\in S$. Take any $x\in H$ and $f,s\in S$. Since $[e,f,s]\in S$, there exists   $t'\in S$ such that  $[t',e,x] = [x,e,[e,f,s]] = [x,f,s]$. Setting $t = [t',e,f] \in S$, one finds that $[t,f,x] = [x,f,s]$, as required.
The second equivalence follows by Lemma~\ref{lem.equal}(3).
\end{proof}

\begin{cor}\label{cor.normal}
Let $S$ be a non-empty subset $S$ of a heap $(H, [- - -])$. The following statements are equivalent.
\begin{blist}
\item  $S$ is a normal sub-heap of $(H, [- - -])$.
\item For all $e\in S$, $S$ is a normal subgroup of $(H,\di_e)$.
\item There exists $e\in S$, such that  $S$ is a normal subgroup of $(H,\di_e)$.
\end{blist}
\end{cor}
\begin{proof}
The statement follows immediately from the definition of a normal sub-heap and Lemma~\ref{lem.normal}.
\end{proof}

Obviously every sub-heap of an Abelian heap is normal. 

\begin{defn}\label{def.gen.heap}
Let $X$ be a non-empty subset of a heap $(H, [- - -])$. The intersection of all sub-heaps containing $X$ is called the {\em sub-heap generated by $X$} and is denoted by $(X)$.
\end{defn}

It is clear that intersection of any family of sub-heaps of $(H, [- - -])$ having at least one element in common is a sub-heap, hence Definition~\ref{def.gen.heap} makes sense. Using the correspondence between heaps and groups, one can construct $(X)$ in the following way. Pick an element $e$ of $X$. Then $(X)$ consists of all finite products
$x_1\di_e x_2\di_e\ldots \di_e x_n$, where $x_i\in X$ or $[e,x_i,e]\in X$. The resulting set does not depend on the choice of $e\in X$.
\es

\bs {Quotient heaps}\label{quotient.heap} 
We start by assigning a  relation to a sub-heap of a heap.
\begin{defn}\label{def.rel}
Given a sub-heap $S$ of $(H, [- - -])$ we define a {\em sub-heap relation} $\sim_S$ on $H$ as follows: $x\sim_S y$ if and only if there exists $s\in S$ such that
\begin{equation}\label{rel}
[x,y,s] \in S.
\end{equation}
\end{defn}

\begin{prop} \label{prop.rel}
Let $S$ be a sub-heap of $(H,[---])$.
\begin{zlist}
\item The relation $\sim_S$ is an equivalence relation.
\item For all $s\in S$, the class of $s$ is equal to $S$.
\item For all $x,y$, $x\sim_S y$ if and only if, for all $s\in S$, $[x,y,s]\in S$.
\item If $S$ is a normal sub-heap, then the  set of equivalence classes $H/S$ is a heap with inherited operation:
\begin{equation}\label{oper.quotient}
[\bar{x},\bar{y},\bar{z}] = \overline{[x,y,z]},
\end{equation}
where $\bar{x}\in H/S$ is the class of $x\in H$, etc.
\end{zlist}
\end{prop}
\begin{proof}
(1) The relation $\sim_S$ is reflexive by the Mal'cev identities. Let us assume that $[x,y,s]=t\in S$ for some $s\in S$. Then using the associativity of the heap operation together with the Mal'cev identity one finds that 
$[y,x,t]  =s \in S,$
so $y\sim_S x$. Finally, take $x,y,z\in H$ such that $x\sim_S y$ and $y\sim_Sz$. Hence there exist $s,t\in S$ such that
\begin{equation}\label{rel.s}
[x,y,s] = s'\in S, \qquad [y,z,t] = t'\in S.
\end{equation}
Since $S$ is a sub-heap, $u = [t,t',s] \in S$, and the associativity of the heap operation together with the Mal'cev identities lead to
$[x,z,u] =s',$
so that $x\sim_S z$.

(2) If $x\in S$, then $[x,s,s] = x\in S$, hence $x\sim_S s$. Conversely, if $x\sim_S s$, then there exist $s',s''\in S$ such that $[x,s,s'] = s''$. By Lemma~\ref{lem.equal}(3), 
$
x = [s'',s',s],
$
and hence $x\in S$ since $S$ is a sub-heap of $H$.

(3) Take any $x,y\in H$ and suppose there exists $s\in S$ such that $[x,y,s] \in S$. Then, for all $t\in S$,
$
[x,y,t] =  [[x,y,s],s,t]\in S,
$
since $S$ is closed under the heap operation. 

(4) For all $x,y,e\in H$,
$
[x,y,e] = x\di_e y^{-1},
$
hence in view of (3) $x\sim_S y$ if, and only if, irrespective of the choice of $e\in S$, $x\di_e y^{-1}\in S$, i.e.\ $x = y\di_e t$, for some $t\in S$. $S$ is a normal sub-heap, hence by Corollary~\ref{cor.normal}, $S$ is a normal subgroup of $(H,\di_e)$. Therefore, $H/S$ is the quotient group with the product denoted by $\di$, and, by statement (3)(b) of Lemma~\ref{lem.heap.group},
$$
\overline{[x,y,z]} = \overline{[x,y,z]_{\di_e}} = \overline{x\di_e y^{-1}\di_e z} = \bar{x}\di \bar{y}^{-1}\di \bar{z},
$$
and thus $\overline{[x,y,z]}$ defines a heap operation on $H/S$ as stated.
\end{proof}

We note in passing that the map
$
\pi_S: H\to H/S,$   $x\mapsto \bar{x},
$
is a heap epimorphism.
\es

\bs {The kernel relation and relative kernels} \label{her.kernel} Following the standard universal algebra treatment (see e.g.\ \cite[Section II.6]{BurSan:uni}) the  {\em kernel} of a heap morphism  $\varphi$ from $(H, [---])$ to $(\tilde{H},   [---])$ is an equivalence relation $\Ker(\varphi)$ on $H$ given as
\begin{equation}\label{kernel.rel}
x \; \Ker(\varphi) \; y \qquad \mbox{if and only if}\qquad \varphi(x) = \varphi(y).
\end{equation}
The set of equivalence classes of  the relation $\Ker(\varphi)$ is a heap with the operation on classes being defined by the operation on their representatives.

 There is an equivalent formulation of the kernel relation which gives rise to a quotient heap by a normal sub-heap as described in Section~\ref{quotient.heap}.

\begin{defn}\label{def.e-kernel}
Let $\varphi$ be a heap homomorphism from $(H, [---])$ to $(\tilde{H},   [---])$ and let $e \in \im \varphi$. The {\em kernel of $\varphi$ relative to $e$} or the {\em $e$-kernel} is the subset $\ker_e(\varphi)$ of $H$ defined as the inverse image of $e$, i.e.\ 
\begin{equation}\label{e-ker}
\ker_e(\varphi) := \varphi^{-1}(e) = \{x\in H\; |\; \varphi(x) =e\}\, .
\end{equation}
\end{defn}

\begin{lem}\label{lem.ker}
Let $\varphi$ be a heap morphism from $(H, [---])$ to $(\tilde{H},   [---])$.
\begin{zlist}
\item For all $e\in \im (\varphi)$, the $e$-kernel $\ker_e(\varphi)$ is a normal sub-heap of $(H, [---])$.
\item For all  $e,e'\in \im (\varphi)$, the $e$-kernels $\ker_e(\varphi)$ and $\ker_{e'}(\varphi)$ are isomorphic as heaps.
\item The relation $\sim_{\ker_e(\varphi)}$ is equal to the kernel relation $\Ker(\varphi)$.
\end{zlist}
\end{lem}
\begin{proof}
(1) Let us take any $z\in \varphi^{-1}(e)$. Then $\ker_e(\varphi)$  is simply the kernel of the group homomorphism $\varphi$ from $(H, \di_z, z)$ to $(\hat{H}, \di_e,e)$, and hence it is a normal subgroup of the former. Therefore, $\ker_e(\varphi)$ is a normal sub-heap of $(H, [---])$  by Corollary~\ref{cor.normal}.

(2) This follows from the group isomorphism in Lemma~\ref{lem.heap.group}~(3)(a). An isomorphism can also be constructed explicitly by using the swap automorphism \eqref{iso.ef} as follows.  Fix $z\in \ker_e(\varphi)$ and $z'\in \ker_{e'}(\varphi)$ and define
$$
\theta : {\ker_e(\varphi)}
\to \ker_{e'}(\varphi), \quad x\mapsto  \tau_z^{z'}(x), 
\qquad \theta^{-1} : {\ker_{e'}}(\varphi)\to \ker_{e}(\varphi), \quad y\mapsto  \tau^z_{z'}(y).
$$

(3) Let us first assume that $\varphi(x) = \varphi(y)$, and let $z\in \varphi^{-1}(e)$. 
Hence $[x, y, z]\in \ker_e(\varphi)$, i.e.\ $x\sim_{\ker_e(\varphi)}y$, by the Mal'cev identities and the definition of a heap homomorphism.

Conversely, if $[x,y,s]\in \ker_e(\varphi)$ for some $s\in \ker_e(\varphi)$, then  $\varphi(s) = e = \varphi\left([x, y, s]\right)$. Since $\varphi$ is a homomorphism of heaps we  thus obtain $e =  [\varphi(x), \varphi(y), e ]$, and therefore  $\varphi(x)=\varphi(y)$ by Lemma~\ref{lem.equal}.
\end{proof}

In view of Lemma~\ref{lem.ker} we no longer need to talk about kernels in relation to a fixed element in the codomain. Therefore we might skip writing $e$ in $\ker_e(\varphi)$, and while saying kernel we mean both the normal sub-heap $\ker(\varphi)$ of the domain or the relation $\Ker(\varphi)$ on the domain. The term $e$-kernel and notation $\ker_e$ are still useful, though, if we want to specify the way the kernel is calculated or we prefer to have equality of objects rather than merely an isomorphism. Lemma~\ref{lem.ker} yields a characterisation of injective homomorphisms.

\begin{cor}\label{cor.ker}
A heap homomorphism $\varphi$ is injective if and only if there exists an element of the codomain with a singleton pre-image, if and only if $\ker (\varphi)$ is a singleton (trivial) heap.
\end{cor}
\begin{proof}
Let $e\in \im\varphi$ be such that $\varphi^{-1}(e) = \{z\}$. If $\varphi(x)=\varphi(y)$, then
$
e 
= \varphi\left([x,y,z]\right),
$
hence $z = [x,y,z]$, and $x=y$ by Lemma~\ref{lem.equal}. The converse and the other equivalence  are clear.
\end{proof}
\es 

\section{Trusses}\label{sec.truss}
This section is devoted to systematic introduction of trusses and two particular sub-structures: ideals, whose definition follows the ring-theoretic intuition, and paragons, which give rise to the truss structure on a quotient heap. In the second part of this section we give some constructions and examples of trusses.
 
\bs{Trusses: definitions}\label{sec.truss.def}
The notions in the following Definition~\ref{def.truss} and Remark~\ref{rem.truss} have been introduced in \cite{Brz:tru}.

\begin{defn}\label{def.truss}
A {\em truss} is an algebraic system consisting of a set $T$, a ternary operation $[---]$ making $T$ into an Abelian heap, and an associative binary operation $\cdot$ (denoted by juxtaposition) which distributes over $[---]$, that is, for all $w,x,y,z\in T$,
\begin{equation}\label{distribute}
w[x, y, z] = [wx, wy, wz], \qquad [x, y, z]w = [xw, yw, zw].
\end{equation}
A truss is said to be {\em commutative} if the binary operation $\cdot$ is commutative. 

Given trusses $(T,[---],\cdot)$ and $(\tilde{T}, [---],\cdot)$ a function $\varphi: T\to \tilde{T}$ that is both a morphism of heaps (with respect of $[---]$) and semigroups (with respect to $\cdot$) is called a {\em morphism of trusses} or a {\em truss homomorphism}.
The category of trusses is denoted by $\truss$.

By a {\em sub-truss} of $(T,[---],\cdot)$ we mean a non-empty subset of $T$ closed under both operations.
\end{defn}

Any singleton set $\{*\}$ has a trivial heap operation $[***] = *$  and a trivial semi-group operation $** =*$ (the only functions with $\{*\}$ as a codomain), which obviously satisfy the truss distributive laws. This is the {\em trivial truss} which we denote by $\star$. For any truss $(T,[---],\cdot)$, the unique function $T\to \{*\}$ is a homomorphism of trusses from $(T,[---],\cdot)$ to $\star$, and thus $\star$ is a terminal object in $\truss$. Global points of objects $(T,[---],\cdot)$ in $\truss$, i.e.\ all truss homomorphisms from $\star$  to $(T,[---],\cdot)$ are in one-to-one correspondence with idempotents in $(T,\cdot)$.

The following lemma follows immediately from the definition of a truss.
\begin{lem}\label{lem.opp}
If $(T,[---],\cdot)$ is a truss, then so is $(T, [---], \cdot^{\mathrm{op}})$, where $\cdot^{\mathrm{op}}$ is a semigroup operation on $T$ opposite to $\cdot$. We refer to $(T, [---], \cdot^{\mathrm{op}})$ as the {\em truss opposite to $(T,[---],\cdot)$} and denote it by $(T^{\mathrm{op}}, [---], \cdot)$ or simply $T^{\mathrm{op}}$.
\end{lem}

\begin{rem}\label{rem.truss}
The notion of a truss can be weakened by not requesting that $(T,[---])$ be an Abelian heap, in which case we will call the system $(T,[---],\cdot)$ a  {\em skew truss} or {\em near truss} or not requesting the two-sided distributivity of $\cdot$ over $[---]$, in which case we will call $(T,[---],\cdot)$ a {\em left} or {\em right} (depending on which of the equations \eqref{distribute} is preserved) (skew) truss. The opposite to a left (skew) truss is a right (skew) truss and vice versa. In the present text we concentrate on trusses with no adjectives as defined in Definition~\ref{def.truss}, although in some places we might point to skew or one-sided generalisations of the claims made.
\end{rem}

By standard universal algebra arguments, the image of a truss homomorphism is a sub-truss of the codomain. We postpone the analysis of relative kernels of homomorphisms until Section~\ref{sec.paragon}, in the meantime we make the following observation on kernels relative to idempotent  elements.

\begin{lem}\label{lem.ker.truss}
Let $\varphi$ be a truss homomorphism from $(T, [---], \cdot)$ to $(\tilde{T},   [---], \cdot)$. If  $e \in \im \varphi$ is an idempotent in $(\tilde{T}, \cdot)$, then the $e$-kernel of $\varphi$ is a sub-truss of the domain.
\end{lem}
\begin{proof}
By Lemma~\ref{lem.ker}, the $e$-kernel is a sub-heap of the domain $(T, [---])$. Since $e$ is an idempotent with respect of the codomain semigroup operation and since a truss homomorphism $\varphi$ respects  semigroup operations, for all $x,y\in T$, if $\varphi(x)=\varphi(y) = e$, then  $\varphi(xy)=e$.
\end{proof}

\bs{The actions}\label{sec.action}
The following proposition is a heap formulation of \cite[Theorem~2.9]{Brz:tru}.
\begin{prop}\label{prop.act}
For a truss $(T,[---],\cdot)$ and an element $e\in T$, define the function
\begin{equation}\label{act}
\lambda^e: T\times T\longrightarrow T, \qquad (x,y)\mapsto [e,xe,xy].
\end{equation}
Then $\lambda^e$ is an action of $(T,\cdot)$ on $(T,[---])$ by heap homomorphisms, i.e., for all $w,x,y,z\in T$
\begin{subequations}\label{acts}
\begin{equation}\label{act.1}
\lambda^e(xy, z) = \lambda^e\left(x,\lambda^e(y,z)\right),
\end{equation}
\begin{equation}\label{act.2}
\lambda^e\left(w,[x,y,z]\right) = \left[\lambda^e(w,x),\lambda^e(w,y),\lambda^e(w,z)\right].
\end{equation}
\end{subequations}
\end{prop}
\begin{proof}
The assertion is proven by direct computation that uses axioms of a truss. Explicitly, to prove \eqref{act.1}, let us take any $x,y,z\in T$ and compute
\begin{align*}
\lambda^e\left(x,\lambda^e(y,z)\right) &=  \left[e, xe, x[e,ye,yz]\right] 
= \left[e,xye,xyz\right] = \lambda^e(xy, z).
\end{align*}

The proof of equality \eqref{act.2} is slightly more involved and, in addition to the truss distributive law and the heap axioms \eqref{heap.axioms}, it uses also \eqref{assoc+} in Lemma~\ref{lem.equal},
\begin{align*}
\left[\lambda^e(w,x),\lambda^e(w,y),\lambda^e(w,z)\right] &= \left[\lambda^e(w,x),[e,we,wy],\lambda^e(w,z)\right]\\
&= \left[\lambda^e(w,x), wy, \left[we, e,[e,we,wz]\right]\right]\\
&=  \left[[e,we,wx], wy, wz\right]\\
&=  \left[e,we,[wx, wy, wz]\right]
 = \lambda^e\left(w,[x,y,z]\right),
\end{align*}
as required.
\end{proof}
\begin{rem}\label{rem.act.skew}
It is worth observing that the arguments of the proof of Proposition~\ref{prop.act} do not use the Abelian property of $[---]$ nor the right truss distributive law. Therefore, the assertions remain true also for left skew trusses. Furthermore, the action $\lambda^e$ has a companion action, also defined for all $e\in T$, 
\begin{equation}\label{act.hat}
\hat\lambda^e: T\times T\longrightarrow T, \qquad (x,y)\mapsto [xy,xe,e].
\end{equation}
Obviously in the case of Abelian $[---]$, $\lambda^e = \hat\lambda^e$, but in the case of left skew trusses the actions may differ.

We also note in passing that the Mal'cev identities imply that, for all $x\in T$, 
\begin{equation}\label{act.absorb}
\lambda^e(x,e) = \hat\lambda^e(x,e) =e.
\end{equation}
\end{rem}

\begin{rem}\label{rem.act}
Proposition~\ref{prop.act} has also a right action version, i.e.\ for all $e\in T$, the function
\begin{equation}\label{act.right}
\varrho^e: T\times T\longrightarrow T, \qquad (x,y)\mapsto [e,ey,xy],
\end{equation}
gives the right action of $(T,\cdot)$ (or the left action of the semigroup opposite to $(T,\cdot)$) on $(T,[---])$ by heap homomorphisms. Also in that case, for all $x\in T$, $\varrho^e(e,x) = e$.
\end{rem}

\bs{Unital and ring-type trusses}\label{sec.bra.ring}
Trusses interpolate between rings and braces introduced in \cite{Rum:bra}, \cite{CedJes:bra} (and skew trusses interpolate between near-rings and skew braces introduced in \cite{GuaVen:ske}). 

\begin{defn}\label{def.braceable}
A truss $(T,[---],\cdot)$ is said to be {\em unital}, if $(T,\cdot)$ is a monoid (with the identity denoted by $1$). 

An element  $0$ of a  truss $(T,[---],\cdot)$ is 
called an {\em absorber} if, for all $x\in T$,
\begin{equation}\label{zero}
 x0 =0 = 0x.
\end{equation}
\end{defn}
Note that an absorber is unique if it exists.
\begin{lem}\label{lem.braceable}
Let $(T,[---],\cdot)$ be a truss.
\begin{zlist}
\item If $T$ is unital, then the operations $+_1$ and $\cdot$ satisfy the left and right brace-type distributive laws, i.e.\ for all $x,y,z\in T$,
\begin{equation}\label{distribute.brace}
x(y+_1 z) = (xy)-_1 x+_1 (xz), \qquad (y+_1 z)x = (yx)-_1 x+_1 (zx).
\end{equation}
\item If $0$ is an absorber of $T$, then $(T,+_0,\cdot)$ is a ring.
\end{zlist}
\end{lem}
\begin{proof}
Both statements are proven by direct calculations. In the case of (1),
\begin{align*}
x(y+_1 z) = x[y, 1, z]  &= [xy, x, xz] = \left[[xy, 1, 1], x, [1,1,xz]\right]\\
&= \left[xy, 1,\left[ [1, x, 1],1,xz\right]\right] 
=(xy)-_1 x+_1 (xz),
\end{align*}
be the heap axioms, the definition of  $+_1$, formula \eqref{inv.e} and the distributivity and unitality.
 The right brace distributivity is proven in a similar way.

To prove assertion (2), take any $x,y,z\in T$, and compute
\begin{align*}
x(y+_0 z) &= x[y, 0, z]  = [xy, 0, xz] = (xy)+_0(xz),
\end{align*}
where, apart from the definition of $+_0$, the truss distributivity \eqref{distribute} and the property $x0=0$ have been used. The right distributive law is proven in a similar way.
\end{proof}

\begin{cor}\label{cor.brace}
 Every truss $(T,[---],\cdot)$ in which $(T,\cdot)$ is a group gives rise to a two-sided brace $(T,+_1,\cdot)$. Conversely, any two-sided brace $(T, +,\cdot)$ gives rise to a unital truss  $(T,[---]_+,\cdot)$ (in which $(T,\cdot)$ is a group). 
\end{cor}

\begin{rem}\label{rem.bra.ring}
In view of Lemma~\ref{lem.braceable} we can informally say of unital trusses that they are {\em braceable}. A truss with the absorber might be referred to as being {\em ring-type}.

Lemma~\ref{lem.braceable} and Corollary~\ref{cor.brace} allow one to view categories of rings and two-sided braces as full subcategories of the co-slice category $(\star\!\downarrow\!\truss)$. The intersection of these subcategories is trivial: (up to isomorphism) it contains only $\star$ understood as the unique morphism $\star\to\star$.
\end{rem}

\begin{lem}\label{lem.brace.im}
Let $\varphi$ be a truss homomorphism from $(T,[---],\cdot)$ to $(\tilde{T},[---],\cdot)$. If $T$ is ring-type or unital then so is $\im \varphi$.
\end{lem}
\begin{proof}
Let $0$ be the absorber of $T$, then, since truss homomorphisms preserve binary operations  $\varphi (0)$ is the absorber of $\im \varphi$.
\end{proof}

\begin{rem}\label{rem.bra.rin}
In spirit of Lemma~\ref{rem.truss}, one can consider also one-sided braceable or ring-type trusses. A truss $(T,[---], \cdot)$ is {\em left braceable} (resp.\ {\em right braceable}) if $(T,\cdot)$ has a right (resp.\ left) identity. A truss $T$ is of {\em  left ring-type} (resp.\ {\em right ring-type}) provided it has a right (resp.\ left) absorber, meaning an element $z$ such that only the first (resp.\  the second) equality in \eqref{zero} holds.

If the truss is left braceable, then the construction of Lemma~\ref{lem.braceable} yields a left brace (and the right braceability leads to a right brace). Similarly, a left ring-type truss yields a left near-ring.
\end{rem}

Finally, if a truss contains a central element (with respect to the semi-group operation), then one can associate a ring to it.
\begin{lem}\label{lem.ring.cen}
Let $(T,[---], \cdot)$ be a truss and let $e\in T$ be central in $(T,\cdot)$. Define the binary operation $\bullet_e$ on $T$ by,
\begin{equation}\label{bullet}
x\bullet_e y = [xy, [x,e,y]e, e], 
\end{equation}
for all $x,y\in T$. Then $(T,+_e,\bullet_e)$ is a ring. 
\end{lem}
\begin{proof}
In terms of the binary group operation $+_e$, the operation \eqref{bullet} reads
$$
x\bullet_e y = xy -_e (x+_ey)e.
$$ 
Since $e$ is the zero for $+_e$ and it is a central element of $(T,\cdot)$ all assumptions of \cite[Theorem~5.2]{Brz:tru} are satisfied, and the assertion follows by \cite[Theorem~5.2]{Brz:tru} (or by \cite[Corollary~5.3]{Brz:tru}).
\end{proof}

\bs{Paragons}\label{sec.paragon}
The question we would like to address in the present section is this: what conditions should a sub-heap $S$  of a truss $T$ satisfy so that the multiplication  descends to the quotient sub-heap $T/S$? In response we propose the following
\begin{defn}\label{def.paragon}
A {\em left paragon} of a truss $(T,[---],\cdot)$ is a sub-heap $P$ of $(T,[---])$ such that, for all $p\in P$, $P$ is closed under the left action $\lambda^p$ in Proposition~\ref{prop.act}, i.e.\ $\lambda^p(T\times P)\subseteq P$.

A sub-heap $P$ that, for all $p\in P$ is closed under the right action $\varrho^p$ in Remark~\ref{rem.act}, i.e.\ $\varrho^p(P\times T)\subseteq P$, is called a {\em right paragon}.

A sub-heap that is both left and right paragon is called a {\em paragon}
\end{defn}

Although, as is discussed in more detail  in Remark~\ref{rem.paragon.brace}, an ideal in a brace is a paragon in the corresponding left (or right) truss and as observed below a paragon in a ring-type truss containing the absorber is an ideal in the corresponding ring, we use the term `paragon' to differentiate it from a closer to ring-theoretic intuition notion of an ideal proposed in Section~\ref{sec.ideal}.

\begin{rem}\label{rem.paragon}
Written explicitly, conditions for a sub-heap $P$ of $(T,[---],\cdot)$ to be a paragon are: for all $x\in T$ and all $p,p'\in P$
\begin{equation}\label{paragon}
[xp,xp',p'] \in P  \qquad \mbox{and} \qquad [px,p'x,p'] \in P.
\end{equation}
The first of equations \eqref{paragon} defines a left paragon, while the second one defines a right paragon.

Note that inclusions \eqref{paragon} are equivalent to,
\begin{equation}\label{paragon+}
[xp',xp,p'] \in P  \qquad \mbox{and} \qquad [p'x,px,p'] \in P.
\end{equation}
Indeed, since $P$ is a sub-heap, if $[xp,xp',p'] \in P$, then
$$
P\ni [p',[xp,xp',p'] ,p'] = [xp',xp,p'],
$$
by \eqref{assoc+} and one of the Mal'cev identities. The converse follows from the equality $[xp,xp',p'] = [p', [xp',xp,p'],p']$. The equivalence of the right paragon identities is proven in a similar way.
\end{rem}

Obviously, $T$ itself is its own paragon.  Furthermore,

\begin{lem}\label{lem.paragon.sing}
Any singleton subset of $(T,[---],\cdot)$ is a paragon in $T$.
\end{lem}
\begin{proof}
This is an immediate consequence of the Mal'cev identities.
\end{proof}

If 0 is the absorber in $(T,[---],\cdot)$, then any paragon that contains 0 is an ideal in the ring $(T,+_0,\cdot)$; simply take $p'=0$ in \eqref{paragon} to deduce that $xp,px\in P$, for all $x\in T$ and $p\in P$.

The definition of a paragon displays the  universal-existential interplay characteristic of heaps.

\begin{lem}\label{lem.paragon.exist}
Let $P$ be a sub-heap of a truss $(T,[---],\cdot)$. The following statements are equivalent.
\begin{zlist}
\item $P$ is a paragon in  $(T,[---],\cdot)$.
\item There exists $e\in P$ such that $P$ is closed under $\lambda^e$ and $\varrho^e$, i.e.\ there exists $e\in P$ such that for all $x\in T$ and $p\in P$,
\begin{equation}\label{paragon.e}
[xp,xe,e] \in P\; \mbox{(equiv.\ $[xe,xp,e] \in P$)} \quad \mbox{and} \quad [px,ex,e] \in P\; \mbox{(equiv.\ $[ex,px,e] \in P$)}.
\end{equation}
\end{zlist}
\end{lem}
\begin{proof}
That statement (1) implies (2) is obvious. In the converse direction, since $P$ is a sub-heap of $(T,[---])$, for all $p,p'\in P$, $[p,p',e]\in P$. Thus the first of inclusions \eqref{paragon.e} implies that
$$
P\ni \left[x[p,p',e],xe,e\right] = \left[[xp,xp',xe],xe,e\right] = \left[xp,xp',e\right],
$$
by the distributive and associative laws and the Mal'cev identity. Again using that $P$ is a sub-heap and the above arguments we obtain,
$$
P\ni \left[[xp,xp',e], e,p'\right] = \left[xp,xp',p'\right],
$$
as required. The second closeness condition for paragons is proven in a similar way. The equivalent formulation of conditions in \eqref{paragon.e} is established as in Remark~\ref{rem.paragon}.
\end{proof}

Occasionally, paragons are closed under multiplication.
\begin{lem}\label{lem.paragon.truss}
A left (resp.\ right) paragon $P$ is a sub-truss of $T$ if and only if there exists $e\in P$ such that, for all $p\in P$, $pe\in P$ (resp.\ $ep \in P$). In  particular, if $T$ is left (resp.\ right) braceable with identity 1, then any left (resp.\ right) paragon containing 1 is closed under multiplication.
\end{lem}
\begin{proof}
If there exists an element $e$ in a left paragon $P$ as specified, then, for all $p,p'\in P$, $[pp',pe,e]\in P$. By Lemma~\ref{lem.equal}(3),
$
pp' =  [[pp',pe, e],e,pe]$, and hence $pp' \in P$,
since $P$ is a sub-heap. The converse is obvious.
\end{proof}

\begin{rem}\label{rem.paragon.brace}
Recall, for example from \cite[Definition~2.8]{CedJes:bra}, that an ideal of a left brace $(B,+,\cdot)$ is defined as a subgroup of $(B,+)$ which is also a normal subgroup of $(B,\cdot)$ and is closed under the action $\lambda^{1}$ (where 1 is a common neutral element of additive and multiplicative groups). Thus an ideal of a left brace is a paragon in the corresponding left truss. 

Lemma~\ref{lem.paragon.truss} indicates that already on the level of unital trusses (no requirement for $(B,\cdot)$ to be a group), paragons containing the identity of $(B,\cdot)$ are closed under the multiplication, i.e.\ they are necessarily submonoids of $(B,\cdot)$.
\end{rem}

\begin{lem}\label{lem.ker.paragon}
Let $\varphi: T\to \tilde{T}$ be a morphism of trusses. For all $e\in \im \varphi$,  the $e$-kernel $\ker_e(\varphi)$  is a paragon of $(T,[---],\cdot)$.
\end{lem}
\begin{proof}
Take any $p,p'\in T$ such that $\varphi(p) = \varphi(p') =e$. Then, for all $x\in T$,
$$
\begin{aligned}
\varphi\left([xp, xp',p']\right) = \left[\varphi(xp), \varphi(xp'),\varphi(p')\right] &= \left[\varphi(x)\varphi(p), \varphi(x)\varphi(p'),e\right]
= e,
\end{aligned}
$$
by the preservation properties of $\varphi$, the choice of $p,p'$ and one of the Mal'cev identities. Therefore, $\ker_e(\varphi)$ is a left paragon. In a similar way one proves that $\ker_e(\varphi)$ is a right paragon. 
\end{proof}

The following proposition gives an answer to the question asked at the beginning of the present section.

\begin{prop}\label{prop.paragon}
Let $P$ be a sub-heap of a truss $(T,[---],\cdot)$. Then the quotient heap $T/P$ is a truss such that the canonical epimorphism $T\to T/P$ is a morphism of trusses if and only if $P$ is a paragon.
\end{prop}
\begin{proof}
Assume that $P$ is a paragon. We need to show that the sub-heap relation $\sim_P$ is a congruence, i.e.\ if $x\sim_P y$ and $x'\sim_P y'$, then $xx'\sim_Pyy'$. By the definition of $\sim_P$, there exist $p,p'\in P$ such that $[x,y,p], [x',y',p']\in P$. Therefore, by \eqref{paragon}, distributive laws and Lemma~\ref{lem.equal}(3),
$
 \left[xx',xy',p'\right] = \left[x[x',y',p'],xp', p'\right] \in P$, 
i.e.\
\begin{equation}\label{comp1}
xx'\sim_P xy'. 
\end{equation}
On the other hand and be the same token
$
\left[xy',yy',p\right] = \left[[x,y,p]y',py', p\right]\in  P$,
and hence
\begin{equation}\label{comp2}
xy'\sim_P yy'.
\end{equation}
Relations \eqref{comp1} and \eqref{comp2} combined with the transitivity of $\sim_P$ yield the assertion. 

Since $\sim_P$ is a congruence relation in $(T,[---],\cdot)$ the binary operation $\cdot$ descends to the quotient heap $(T/P, [---])$, thus leading to the truss structure on $T/P$ such that the canonical map $T\to T/P$ is a homomorphism of trusses. 

In the converse direction, assume that the epimorphism $\pi: T\to T/P$ is a homomorphism of trusses. Then $\ker_P \pi = P$ by Proposition~\ref{prop.rel}(2), and by Lemma~\ref{lem.ker.paragon}, $P$ is a paragon, as required.
\end{proof}

\begin{cor}\label{cor.ker.paragon}
Let $\varphi$ be a morphism of trusses with domain $T$. For all $e\in \im \varphi$, the quotient heaps $T/{\ker_e(\varphi)}$ are mutually isomorphic trusses.
\end{cor}
\begin{proof}
Since $\ker_e(\varphi)$ is a paragon by Lemma~\ref{lem.ker.paragon}, $T/{\ker_e(\varphi)}$ is a truss by Proposition~\ref{prop.paragon}. The independence of the choice of $e$ follows by Lemma~\ref{lem.ker}.
\end{proof}

In case a paragon contains a central element it has a natural interpretation in terms of the ring associated to a truss by Lemma~\ref{lem.ring.cen}.

\begin{lem}\label{lem.ring.paragon}
Let $P$ be a sub-heap of a truss $(T,[---],\cdot)$, and let $e\in P$ be a central element  of the monoid $(T,\cdot)$. If $P$ is a (left or right) paragon in $(T,[---],\cdot)$, then $P$ is a (left or right) ideal in the associated ring $(T,+_e,\bullet_e)$.
\end{lem} 
\begin{proof}
Assume that $P$ is a left paragon. Since $e\in P$, $P$ is a subgroup of $(T,+_e)$. Furthermore, for all $p\in P$, $[e^2,ep,e]\in P$ by Remark~\ref{rem.paragon} (take $x=e$, $p'=e$ in the first of equations \eqref{paragon+}). Take any $x\in T$. With the help of the truss distributive law, the fact that $(T, [---])$ is an Abelian heap, centrality of $e$, and heap properties, one can compute
$$
\begin{aligned}
x\bullet_e p &= [xp,[x,e,p]e, e] =  [xp, xe, [e^2,ep, e]]
 = [[xp, xe,e],e, [e^2,ep, e]].
\end{aligned} 
$$
Since all of the $[xp, xe,e]$, $e$ and $[e^2,ep, e]$ are in $P$, so is their heap bracket, and thus $x\bullet_e p \in P$. This proves that $P$ is a left ideal in the ring $(T,+_e,\bullet_e)$. The case of the right paragon is dealt with in a similar way.
\end{proof}

\subsection{Ideals}\label{sec.ideal}

The definition of an ideal in a truss follows the ring theoretic intuition.
\begin{defn}\label{def.ideal}
An {\em ideal} of a truss $(T,[---],\cdot)$ is a sub-heap $S$ of $(T,[---])$ such that, for all $x\in T$ and $s\in S$,
$xs\in S$ and  $sx\in S$.
\end{defn}

\begin{lem}\label{lem.ideal}
Let $S$ be an ideal of $(T,[---],\cdot)$. Then 
\begin{zlist}
\item $S$ is a paragon.
\item  The sub-heap relation $\sim_S$ is a congruence. 
\item $S$ is the absorber in the quotient truss $(T/S, [---],\cdot)$.
\end{zlist}
\end{lem}
\begin{proof}
(1) Since $S$ is a sub-heap and ideal, 
$
[xs,xs',s']$ and $[sx,s'x,s']\in S$, 
for all $x\in T$ and $s,s'\in S$. 

(2) Follows by Proposition~\ref{prop.paragon} and (1).

(3) This follows immediately from the definition of ideal.
\end{proof}

\begin{lem}\label{lem.ker.ideal}
Let $\varphi: T\to \tilde{T}$ be a morphism of trusses. If $\im \varphi \subseteq \tilde{T}$ has the absorber $0$, then $\ker_0(\varphi)$  is an ideal of $(T,[---],\cdot)$.
\end{lem}
\begin{proof}
For all $x\in T$ and $s\in \ker_0(\varphi)$,
$
\varphi(xs) = \varphi(x)\varphi(s) = \varphi(x)0 = 0,
$
by the absorber property \eqref{zero}. Hence $xs \in \ker_0(\varphi)$. The second condition is proven in a similar way.
\end{proof}

\begin{rem}\label{rem.ideal}
Similarly to the case of rings, one can also talk about one-sided ideals; a left or right ideal in a truss are defined in obvious ways. If $(T,[---],\cdot)$ is a truss arising from a brace, then it has no proper (i.e.\ different from $T$) ideals, just as a division ring has no proper non-trivial ideals.  
\end{rem}

Being paragons, ideals are closed under the actions $\lambda^e$ and $\varrho^e$ discussed in Section~\ref{sec.action}.

\begin{lem}\label{lem.ideal.act}
Let $(T,[---],\cdot)$  be a truss, $S$ a sub-heap of $(T,[---])$, let $e\in S$ and let $\lambda^e$ and $\varrho^e$ denote the (suitable restrictions of) the actions defined in Proposition~\ref{prop.act} and Remark~\ref{rem.act}. Then
\begin{zlist}
\item If $S$ is a left ideal in $(T,[---],\cdot)$, then $\lambda^e(T\times S)\subseteq S$.
\item If $S$ is a right ideal in $(T,[---],\cdot)$, then $\varrho^e(S\times T)\subseteq S$.
\end{zlist}
\end{lem}

The definition of a principal ideal hinges on the following simple lemma.
\begin{lem}\label{lem.inter.ideal}
If $(S_i)_{i\in I}$ is a family of (left, right) ideals in  a truss $(T,[---],\cdot)$ with at least one element in common, then 
$$
S = \bigcap_{i\in I} S_i,
$$
is an (left, right) ideal in $(T,[---],\cdot)$.
\end{lem}

\begin{defn}\label{def.principal}
Let $X$ be a non-empty subset of a $(T,[---],\cdot)$. An (left, right) ideal {\em generated by $X$} is defined as the intersection of all (left, right) ideals containing $X$. If $X =\{e\}$ is a singleton set, then the ideal generated by $X$ is called a {\em principal} ideal and is denoted by $<\!e\!>$ (or $Te$ in the case of left or $eT$ in the case of right ideal).
\end{defn}

In view of the discussion at the end of Section~\ref{sub-heap} a principal ideal $\langle e\rangle$ consists of $e$ and all finite sums $\sum_{i=1}^n{}\!^e\, x_i$ with $x_i = a_ieb_i$ or $x_i = -_ea_ieb_i = [e,a_ieb_i,e]$, for some $a_i,b_i\in T$,
with understanding that $a_i$ or $b_i$ can be null (as in $eb_i$ or $a_ie$).
  Put differently, every element $a$ of $\langle e\rangle$ can be written as 
\begin{equation}\label{principal}
x=[[x_1, x_2, \ldots, x_{2n+1}]]_n,
\end{equation}
where the double-bracket is defined in \eqref{multi} and $x_i = e$ or $x_i = a_ieb_i$, for some $a_i,b_i\in T$.

Principal ideals are used for a universal construction of ring-type trusses.

\begin{prop}\label{prop.uni.ringable}
Let $(T,[---],\cdot)$ be a truss. Then, for all $e\in T$ there exist a ring-type truss $T_e$ and a truss homomorphism $\pi_e: T\to T_e$ such that $\pi_e(e)$ is an absorber in $T_e$, and which have the following  universal property. For all morphisms of trusses  $\psi: T\to \tilde{T}$ that map $e$ into an absorber in $\tilde{T}$ there exists a unique filler (in the category of trusses) of the following diagram
$$
\xymatrix{T \ar[rr]^-{\pi_e} \ar[dr]_\psi&&T_e \ar@{-->}[dl]^-{\psi_e} \\& \tilde{T}. &}
$$
The truss $T_e$ is unique up to isomorphism.
\end{prop}
\begin{proof}
Let $T_e = T/\langle e \rangle$, the quotient of $T$ by the principal ideal generated  by $e$, and let $\pi_e: T\to T_e$ be the canonical surjection, $x\mapsto \bar{x}$. Then $\pi_e(e) = \bar{e}$ is an absorber by Lemma~\ref{lem.ideal}.

Since $\psi(e)$ is an absorber and $\psi$ is a truss morphism, for all $x\in T$,
\begin{equation}\label{psi.absorb.ring}
\psi(x e y) = \psi(x)\psi(e)\psi(y) = \psi(e).
\end{equation}
If $a\in \langle e\rangle $,  then its presentation \eqref{principal} together with \eqref{psi.absorb.ring}, the fact that $\psi$ is a heap morphism and that $[---]$ (and hence any $[[-\ldots-]]$) is an idempotent operation imply that $\psi(a) = \psi(e)$. Hence, if $x\sim_{\langle e\rangle} y$, i.e.\ there exist $a,b\in \langle e\rangle$ such that
$$
[x,y,a] = b,
$$
then
$$
[\psi(x),\psi(y), \psi(e)] = [\psi(x),\psi(y), \psi(a)] = \psi\left([x,y,a]\right) = \psi(b) = \psi(e).
$$
Therefore, $\psi(x)= \psi(y)$ by Lemma~\ref{lem.equal}, and thus we can define the function 
$$
\psi_e: T_e \lra \tilde{T}, \qquad \bar{x}\lto \psi(x).
$$
Since $\psi$ is a morphism of heaps, so is $\psi_e$. By construction, $\psi_e\circ \pi_e = \psi$. The uniqueness of both $\psi_e$ and $T_e$ is clear (the latter by the virtue of the universal property by which $T_e$ is defined).
\end{proof}

\begin{rem}\label{rem.para.inter}
Any non-empty intersection of paragons in a truss $T$ is also a paragon, hence one can define paragons generated by a subset $X$ as intersection of all paragons containing $X$, as in Definition~\ref{def.principal}. Note, however, that a `principal' paragon, i.e.\ a paragon generated by a singleton set, is equal to this set, since every singleton subset of $T$ is a paragon by Lemma~\ref{lem.paragon.sing}.
\end{rem}

\subsection{An Abelian heap as a truss}\label{sec.heap.truss}
In this and the following sections we present a number of examples of trusses arising from an Abelian heap.

\begin{lem}\label{lem.disc.truss}
Let $(H,[---])$ be an Abelian heap and let $e\in H$. Then the binary operation, for all $x,y\in H$,
\begin{equation}\label{disc.truss}
x\cdot_e y =e,
\end{equation}
makes $(H,[---])$ into a (commutative, ring-type) truss. 
\end{lem}
\begin{proof}
Clearly \eqref{disc.truss} is an associative operation. It distributes over $[---]$ by the idempotent property of heap operations.
\end{proof}

Since the truss in Lemma~\ref{lem.disc.truss} has an absorber that absorbs all products, we might refer to such a truss as  being {\em fully-absorbing}.

\begin{lem}\label{lem.heap.truss}
Let $(H,[---])$ be an Abelian heap and let $\alpha$ be an idempotent endomorphism of $(H,[---])$. Define the binary operations
\begin{equation}\label{alpha}
x\cdot_\alpha y = [x,\alpha(x),y] \qquad \mbox{and}\qquad x\hat\cdot_\alpha y = [x,\alpha(y),y] 
\end{equation}
for all $x,y\in H$. Then $(H,[---], \cdot_\alpha)$  and $(H,[---], \hat\cdot_\alpha)$ are trusses.
\end{lem} 
\begin{proof}
First we need to check that the operation defined in \eqref{alpha} is associative. For all $x,y,z\in H$,
$$
\begin{aligned}
(x\cdot_\alpha y)\cdot_\alpha z &= \left[\left[x,\alpha(x), y\right], \alpha\left(\left[x,\alpha(x), y\right]\right), z\right]\\
&= \left[\left[x,\alpha(x), y\right],\left[ \alpha\left(x\right), \alpha\left(\alpha(x)\right),  \alpha\left(y\right)\right], z\right]\\
&= \left[\left[x,\alpha(x), y\right],  \alpha\left(y\right), z\right]
= \left[x,\alpha(x), \left[y,  \alpha\left(y\right), z\right]\right] = x\cdot_\alpha (y\cdot_\alpha z),
\end{aligned}
$$
where the second equality follows by the endomorphism property of $\alpha$, the third one is a consequence of the fact that $\alpha$ is an idempotent and the Mal'cev identity. The penultimate equality follows by the associative law of heaps.

Next we need to check the distributive laws. For all $w,x,y,z\in H$,
$$
\begin{aligned}
 {}[w\cdot_\alpha x,w\cdot_\alpha y,w\cdot_\alpha z] &= \left[\left[w,\alpha(w),x\right], \left[w,\alpha(w),y\right], \left[w,\alpha(w),z\right]\right]\\
 &= \left[[w,w,w], [\alpha(w),\alpha(w),\alpha(w)], [x,y,z]\right]\\
 &= \left[w,\alpha(w),  [x,y,z]\right] = w\cdot_\alpha[x,y,z],
\end{aligned}
$$
by Lemma~\ref{lem.equal}(4) and the Mal'cev identity. The other distributive law follows by similar arguments and by the fact that $\alpha$ preserves heaps operations.

The second multiplication is the opposite of the first one.
\end{proof}

\begin{cor}\label{cor.heap.truss}
Any Abelian heap $H$ with either of the binary operations
$$
xy=y \qquad \mbox{or} \qquad xy=x,
$$
for all $x,y\in H$,
is a right (in the first case) or left (in the second case) braceable truss.
\end{cor}
\begin{proof}
The operations are obtained by setting $\alpha = \id$ in Lemma~\ref{lem.heap.truss}. The right or left braceability is obvious.
\end{proof}

\begin{cor}\label{cor.heap.truss.+}
Let $(H,[---])$ be an Abelian heap and let $e\in H$. Then $(H,[---])$ together with the binary operation, for all $x,y\in H$,
$$
x\cdot_e y = [x,e,y] = x+_e y,
$$
is a commutative braceable truss (in fact a two-sided brace).
\end{cor}
\begin{proof}
Set $\alpha(x) =e$, for all $x\in H$ in Lemma~\ref{lem.heap.truss}.
\end{proof}

\subsection{The endomorphim truss}\label{sec.end}
A set of all endomorphisms of an Abelian heap can be equipped with the structure of a truss.
\begin{prop}\label{prop.end}
Let $(H,[---])$ be an Abelian heap. The set $E(H)$ of all endomorphisms of $(H,[---])$ is a truss with the pointwise heap operation, for all $\alpha,\beta,\gamma \in E(H)$, 
\begin{equation}\label{ter.end}
[\alpha, \beta,\gamma] : H\to H, \qquad x\mapsto  [\alpha(x), \beta(x),\gamma(x)],
\end{equation}
 and the composition $\circ$ of functions.
\end{prop}
\begin{proof}
First we need to check that, for all  $\alpha,\beta,\gamma \in E(H)$, $[\alpha, \beta,\gamma]$ is a homomorphism of heaps. To this end, let us take any $x,y,z \in H$ and, using Lemma~\ref{lem.equal}(4), compute
\begin{align*}
[\alpha,\beta,\gamma]\left([x,y,z]\right) &= \left[\alpha\left([x,y,z]\right), \beta\left([x,y,z]\right), \gamma\left([x,y,z]\right)\right]\\
&= \left[\left[\alpha(x),\alpha(y),\alpha(z)\right], \left[\beta(x),\beta(y),\beta(z)\right],\left[ \gamma(x), \gamma(y), \gamma(z)\right]\right]\\
&=\left[\left[\alpha(x),\beta(x),\gamma(x)\right], \left[\alpha(y),\beta(y),\gamma(y)\right],\left[ \alpha(z), \beta(z), \gamma(z)\right]\right]\\
&= \left[\left[\alpha,\beta,\gamma\right](x), \left[\alpha,\beta,\gamma\right](y),\left[ \alpha, \beta, \gamma\right](z)\right],
\end{align*}
where the definition of the ternary operation on $E(H)$ has been used a  number of times. Therefore, the operation $[---]$ defined by \eqref{ter.end} is well-defined as claimed.

That $E(H)$ with operation \eqref{ter.end} is a heap and that the composition right distributes over \eqref{ter.end} follows immediately from the fact that $(H,[---])$ is a heap and the pointwise nature of definition \eqref{ter.end}. The left distributive law is a consequence of the preservation of the heap ternary operation by a heap homomorphism. 
\end{proof}

\begin{lem}\label{lem.end.bra.rin}
The endomorphism truss is unital and right ring-type.
\end{lem}
\begin{proof}
Obviously  $(E(H),\circ)$ is a monoid since the identity morphism on $H$ is the identity for the composition.  With respect to the heap operation $[---]$ every element of $H$ is an idempotent (by Mal'cev identities), hence any constant function on $H$ is a homomorphism of heaps, which has the left absorber property \eqref{zero} with respect to the composition.
\end{proof}

\begin{lem}\label{lem.sub.end}
Let $(H,[---])$ be an Abelian heap. For all $e\in H$, the endomorphism monoid of the associated group, $\mathrm{End}(H,+_e)$ is a sub-truss of $E(H)$. Furthermore, different choices of $e$ lead to isomorphic sub-trusses of $E(H)$.
\end{lem}
\begin{proof}
Since all elements of $\mathrm{End}(H,+_e)$ preserve $e$, we obtain, for all $\alpha,\beta,\gamma \in \mathrm{End}(H,+_e)$ and $x,y\in H$,
$$
\begin{aligned}
{}[\alpha,\beta,\gamma](x+_e y) &= [\alpha,\beta,\gamma]\left([x,e,y]\right) \\
&=\left[\left[\alpha(x),\beta(x),\gamma(x)\right], \left[\alpha(e),\beta(e),\gamma(e)\right],\left[ \alpha(y), \beta(y), \gamma(y)\right]\right]\\
&= \left[\left[\alpha,\beta,\gamma\right](x), e,\left[ \alpha, \beta, \gamma\right](y)\right] = \left[\alpha,\beta,\gamma\right](x) +_e \left[ \alpha, \beta, \gamma\right](y),
\end{aligned}
$$
by the same arguments as in the proof of Proposition~\ref{prop.end} and by the idempotent property of $[---]$. Hence $\mathrm{End}(H,+_e)$ is a sub-heap of $E(H)$. Obviously, $\mathrm{End}(H,+_e)$ is closed under the composition.

For different $e,f\in H$, the groups $(H,+_e)$, $(H,+_f)$ are isomorphic by Lemma~\ref{lem.heap.group}, hence also the sets $\mathrm{End}(H,+_e)$, $\mathrm{End}(H,+_f)$ are isomorphic with the bijection
$$
\begin{aligned}
\vartheta  = \tau_e^f\circ \alpha\circ \tau_f^e : \mathrm{End}(H,+_e) &\lra \mathrm{End}(H,+_f),\\
 \alpha &\lto [x\mapsto \alpha(x-_e f) +_e f].
\end{aligned}
$$
Since the swap automorphism $\tau_e^f$ (see \eqref{iso.ef}) is a heap homomorphism, so is $\vartheta$. One easily checks that 
$$
\vartheta(\alpha\circ\beta) = \vartheta(\alpha)\circ\vartheta( \beta),
$$ 
for all $\alpha,\beta,\gamma \in \mathrm{End}(H,+_e)$, i.e.\ that $\vartheta$ is an isomorphism of trusses as stated.
\end{proof}

\begin{lem}\label{lem.end.ideal}
Let $S$ be a left ideal in a truss $(T,[---],\cdot)$. The maps
\begin{equation}\label{end.ideal}
\begin{aligned}
\pi_S: &\; T \lra E(S), \qquad x\lto [s\mapsto xs],\\
 \pi_S^\circ : &\; T^{\mathrm{op}} \lra E(S), \quad x\lto [s\mapsto sx],
 \end{aligned}
\end{equation}
are homomorphisms of trusses.
\end{lem}
\begin{proof} The left distributive law, i.e.\ the first of equations \eqref{distribute}, and the definition of an ideal ensure that, for all $x\in T$, the map $\pi_S(x)$ is an endomrphism of the heap $(S,[---])$. The map $\pi_S$ is a homomorphism of heaps by the right distributive law, i.e.\ the second of equations \eqref{distribute}. Finally, the associativity of the product $\cdot$ yields that, for all $x,y\in T$, $\pi_S(xy) = \pi_S(x)\circ\pi_S(y)$. The fact that $\pi_S^\circ$ is a morphism of trusses is proven in a similar way.
\end{proof}

Since a truss is its own ideal we obtain
\begin{cor}\label{cor.end.ideal}
Let $(T,[---],\cdot)$ be a truss. The maps
$$
\begin{aligned}
\pi_T:  &\; T \lra E(T), \qquad x\lto [y\mapsto xy], \\
 \pi_T^\circ :  &\; T^{\mathrm{op}} \lra E(T), \quad x\lto [y\mapsto yx], 
 \end{aligned}
$$
are homomorphism of trusses. If $T$ is unital, then these maps are monomorphisms.
\end{cor}
\begin{proof}
The first statement is contained in Lemma~\ref{lem.end.ideal}. If $T$ is unital with identity $1$, for all $x\in T$, $\pi_T(x)(1) = x = \pi_T^\circ(x)(1)$, hence both $\pi_T$ and $\pi_T^\circ$ distinguish between elements of $T$.
\end{proof}

In Section~\ref{sec.module.paragon} we will also show that one can construct a truss homomorphism from $T$ to the endomorphism truss of any paragon in $T$.

\subsection{The endomorphism truss and the semi-direct product}\label{sec.con.end}
In \cite{Cer:ter} Certaine has observed that the group of automorphisms of a heap  is isomorphic to the holomorph of any group associated to this heap. In this section we extend this observation to endomorphisms of heaps and then apply it to the endomorphism truss.

\begin{lem}\label{lem.end.end}
Let $(H,[---])$ be a heap. For any element $e\in H$, denote by $\mathrm{End}(H,\di_e)$ the monoid of endomorphisms of the associated group $(H,\di_e,e)$. Then
$$
\mathrm{End}(H, [---]) \cong H\times \mathrm{End}(H,\di_e).
$$
\end{lem}
\begin{proof}
Let 
$$
\ell^e: H \lra \mathrm{End}(H,\di_e), \qquad x\lto [y\mapsto [x,e,y] = x\di_e y ].
$$ 
be the left translation map, and consider the map
\begin{equation}\label{iso.1}
\Theta: H\times \mathrm{End}(H,\di_e) \lra \mathrm{End}(H, [---]) , \qquad (x,\alpha) \lto {}_x\theta_\alpha = \ell^e(x)\circ\alpha.
\end{equation}
Written  in terms of the binary group operation $\di_e$, ${}_x\theta_\alpha (y) = x\di_e \alpha(y)$.  Keeping in mind that $[y,z,w] = y\di_e z^{-1}\di_e w$ (see Lemma~\ref{lem.heap.group}) and that $\alpha$ is an endomorphism of $(H,\di_e)$, one easily checks that  ${}_x\theta_\alpha$ is  an endomorphism of $(H,[---])$.

In the converse direction, define 
\begin{equation}\label{iso.2}
\overline{\Theta} :  \mathrm{End}(H, [---])  \lra H\times \mathrm{End}(H,\di_e), \quad \varphi\lto (\varphi(e),\ell^e(\varphi(e)^{-1})\circ \varphi).
\end{equation}
Since $\varphi$ is an endomorphism  of $(H,[---])$ and $\ell^e(\varphi(e)^{-1})\circ \varphi(e) = e$, the second entry in the pair  \eqref{iso.2} is  an endomorphism of $(H,\di_e)$ by Lemma~\ref{lem.heap.group}. 

In view of the definition of ${}_x\theta_\alpha$ in \eqref{iso.1}, for all $(x,\alpha)\in H\times \mathrm{End}(H,\di_e)$,
$$
\begin{aligned}
\overline{\Theta}(\Theta(x,\alpha)) &= ((\ell^e(x)\circ \alpha)(e), \ell^e((\ell^e(x)\circ \alpha)(e)^{-1})\circ \ell^e(x)\circ \alpha)\\
&= (x, \ell^e({x}^{-1})\circ \ell^e(x)\circ \alpha) = (x,\alpha).
\end{aligned}
$$
On the other hand, for all $\varphi \in  \mathrm{End}(H, [---])$, $x\in H$,
$$
\begin{aligned}
\Theta(\overline{\Theta}(\varphi))(x) &= {}_{\varphi(e)}\theta_{\ell^e(\varphi(e)^{-1})\circ \varphi}(x)
= \varphi(e)\di_e (\ell^e(\varphi(e)^{-1})\circ \varphi)(x)\\
& = \varphi(e)\di_e \varphi(e)^{-1}\di_e \varphi(x) = \varphi(x),
\end{aligned}
$$
i.e.\ $\overline{\Theta}$ is the inverse of $\Theta$, as required.
\end{proof}

If $(H,[---])$ is an Abelian heap, the truss structure of $\mathrm{End}(H, [---])$ can be transferred through $\Theta$ to $H\times \mathrm{End}(H,\di_e)$.

\begin{prop}\label{prop.end.end}
Let $(H,[---])$ be an Abelian heap. For any element $e\in H$,  $H\times \mathrm{End}(H,+_e)$ is a truss, isomorphic to $E(H)$, with the product heap structure and the semi-direct product monoid operation, for all $(x,\alpha),(y,\beta)\in H\times \mathrm{End}(H,+_e)$,
\begin{equation}\label{semi-dir}
 (x,\alpha)(y,\beta)  := (x+_e \alpha(y), \alpha\circ\beta) = ([x,e , \alpha(y)], \alpha\circ\beta).
\end{equation}
We denote this truss by  $H\rtimes \mathrm{End}(H,+_e)$.
\end{prop}
\begin{proof}
With the help of isomorphism $\Theta$ in Lemma~\ref{lem.end.end}, the endomorphism truss structure can be transferred  to $H\times \mathrm{End}(H,+_e)$. Explicitly, for all $(x,\alpha), (y,\beta), (z,\gamma) \in H\times \mathrm{End}(H,+_e)$,
\begin{subequations}\label{semi-dir.str}
\begin{equation}\label{semi-dir.heap}
[(x,\alpha), (y,\beta), (z,\gamma)] = \Theta^{-1}\left([\Theta(x,\alpha), \Theta(y,\beta), \Theta(z,\gamma)]\right),
\end{equation}
\begin{equation}\label{semi-dir.mon}
(x,\alpha)(y,\beta) = \Theta^{-1}\left(\Theta(x,\alpha)\circ\Theta(y,\beta)\right).
\end{equation}
\end{subequations}
Our task is to identify operations defined in \eqref{semi-dir.str}. First, note that for all $w\in H$,
$$
\begin{aligned}
 \left[{}_x\theta_\alpha, {}_y\theta_\beta, {}_z\theta_\gamma\right] (w)  = &\left[x+_e\alpha(w), y+_e\beta(w), z+_e\gamma(w)\right]\\
& = \left[x,y,z\right] +_e \left[\alpha(w), \beta(w), \gamma(w)\right]\\
 &= {}_{[x,y,z]}\theta_{[\alpha,\beta,\gamma]}(w).
\end{aligned}
$$
Hence,
$$
\begin{aligned}
\Theta^{-1}\left([\Theta(x,\alpha), \Theta(y,\beta), \Theta(z,\gamma)]\right) &= \Theta^{-1}\left(\left[{}_x\theta_\alpha, {}_y\theta_\beta, {}_z\theta_\gamma\right]\right)
= \Theta^{-1}\left({}_{[x,y,z]}\theta_{[\alpha,\beta,\gamma]}\right)\\
& = 
\Theta^{-1}\left(\Theta\left([x,y,z],{[\alpha,\beta,\gamma]}\right)\right)
= \left([x,y,z],{[\alpha,\beta,\gamma]}\right),
\end{aligned}
$$
i.e.\ equation \eqref{semi-dir.heap} describes the product heap structure on $H\times \mathrm{End}(H,+_e)$.

Next, take any $(x,\alpha), (y,\beta) \in H\times \mathrm{End}(H,+_e)$ and, using the fact that the group homomorphisms preserve neutral elements, compute
$$
({}_x\theta_\alpha\circ {}_y\theta_\beta)(e)  = (\ell^e(x)\circ \alpha\circ\ell^e(y)\circ\beta)(e) = (\ell^e(x)\circ \alpha)(y) = x+_e\alpha(y).
$$
This yields,
$$
\begin{aligned}
\Theta^{-1}\left(\Theta(x,\alpha)\circ\Theta(y,\beta)\right) &= \Theta^{-1}\left({}_x\theta_\alpha\circ {}_y\theta_\beta\right) \\
&= \left(x+_e\alpha(y), \ell^e(-_ex-_e\alpha(y))\circ \ell^e(x)\circ \alpha\circ\ell^e(y)\circ\beta\right).
\end{aligned}
$$
Evaluating the second element of the above pair at $z\in H$ and  using that $\alpha$ is a group homomorphism we find,
$$
\begin{aligned}
\ell^e(-_ex-_e\alpha(y))&\circ \ell^e(x)\circ \alpha\circ\ell^e(y)\circ\beta(z)\\
  &= 
  \left(\ell^e(-_ex-_e\alpha(y))\circ \ell^e(x)\circ \alpha\circ\ell^e(y)\right)(\beta(z))\\
  &=  \left(\ell^e(-_ex-_e\alpha(y))\circ \ell^e(x)\right)(\alpha(y +_e\beta(z)))\\
  &= \ell^e(-_ex-_e\alpha(y))\left(x+_e \alpha(y) +_e\alpha\circ \beta(z))\right) = \alpha\circ \beta(z).
\end{aligned}
$$
Therefore,
$$
\Theta^{-1}\left(\Theta(x,\alpha)\circ\Theta(y,\beta)\right) = \left(x+_e\alpha(y), \alpha\circ\beta\right),
$$
as required.
\end{proof}

We note in passing that since the  endomorphism truss $E(H)$  is independent of choice of any element, the semi-direct product truss $H\rtimes \mathrm{End}(H,+_e)$ is likewise independent on the choice of $e$. The description of the endomorphism truss in terms of the semi-direct product in Proposition~\ref{prop.end.end} gives one an opportunity to construct explicit examples of trusses from groups.

\begin{cor} \label{cor.end.end.1}
Let $(H, +,0)$ be an Abelian group and let $S$ be any subset of $\mathrm{End}(H,+)$ closed under composition of functions and under the ternary operation on $\mathrm{End}(H,+)$  induced from $[---]_+$ on $H$. Then $H\times S$ is a truss with the product heap operation and semi-direct product binary operation.
\end{cor}
\begin{proof} 
This follows immediately by observing that, for all $e$ in a heap $H$,  $\mathrm{End}(H,+_e)$ is a sub-semi-group of $H\rtimes \mathrm{End}(H,+_e)$ and from Lemma~\ref{lem.heap.group} that connects groups with heaps.
\end{proof}

\begin{cor} \label{cor.end.end.2}
Let $(H, +,0)$ be an Abelian group, and let $\alpha$ be an idempotent endomorphism of $(H, +,0)$. Then, for all $a\in \ker \alpha$, $H$ is a truss with the heap operation $[---]_+$ and multiplications, for all $x,y\in H$,
\begin{equation}\label{end.idem}
xy = x+y-\alpha(y)-a \qquad \mbox{or} \qquad xy = x+y-\alpha(x)-a  .
\end{equation}
\end{cor}
\begin{proof}
Through the correspondence of Proposition~\ref{prop.end.end} idempotent endomorphisms of $(H, [---]_+)$ are in one-to-one correspondence with pairs $a\in H$, $\alpha\in \mathrm{End}(H,+)$ such that
$$
a+\alpha(a) =a, \qquad \alpha^2=\alpha.
$$
Hence any pair $(a,\alpha)$ satisfying the hypothesis gives rise to an idempotent endomorphism of $(H, [---]_+)$ and thus there are truss structures as in Lemma~\ref{lem.heap.truss}. Translating pairs $(a,\alpha)$ back into a single map of $H$ through Lemma~\ref{lem.end.end} one obtains the formulae \eqref{end.idem}.
\end{proof}

\begin{rem}\label{rem.end.end.2}
For $a=0$, the second of the truss structures described in Corollary~\ref{cor.end.end.2} is a special case of that in Corollary~\ref{cor.end.end.1}. Take $S=\{\id-\alpha\}$. Being  a singleton set, $S$ is a heap (and, in particular, a sub-heap of $(\mathrm{End}(H,+), [---]_+)$), and since $\alpha$ is an idempotent, so is $\id-\alpha$,  and hence $S$ is closed under the composition. The assertion then follows by Corollary~\ref{cor.end.end.1} through the canonical isomorphism $H\times S\cong H$.
\end{rem}

\subsection{Examples of trusses arising from the semi-direct product construction}\label{sec.ex.semi-dir}
In this section we list a handful of examples that resulting from the discussion presented in Section~\ref{sec.con.end}.

\begin{ex}\label{ex.integers}
Consider the additive group of integers, $(\ZZ, +)$. The endomorphisms of $(\ZZ,+)$ are in one-to-one correspondence with the elements of $\ZZ$, since any $\alpha \in \mathrm{End}(\ZZ,+)$ is fully determined by $\alpha(1) \in \ZZ$. The composition of endomorphisms translates to the product of determining elements. Taking this into account, we can identify $\ZZ\rtimes \mathrm{End}(\ZZ,+)$  with $\ZZ\times \ZZ$ with the heap operation and product, for all $(a_1,a_2), (b_1,b_2), (c_1,c_2)\in \ZZ\times \ZZ$,
$$
\begin{aligned}
& [(a_1,a_2), (b_1,b_2), (c_1,c_2)] = (a_1-b_1+c_1, a_2-b_2+c_2),\\ 
& (a_1,a_2)(b_1,b_2) = (a_1 +a_2b_1, a_2b_2),
\end{aligned}
$$
by Proposition~\ref{prop.end.end}.
This truss  has a left absorber $(0,0)$ as well as, being isomorphic to the endomorphism truss, it is unital with the identity $(0,1)$.
\end{ex}

\begin{ex}\label{ex.v4}
Consider the Klein group (written additively), $V_4=\{0,a,b,a+b\}$. The map given as
$$
\alpha: 0\mapsto 0, \quad a\mapsto 0, \quad b\mapsto b, \quad a+b\mapsto b,
$$
is an idempotent group homomorphism. Consequently, choosing $a$ as a specific element of $\ker\alpha$ in the first of equations \eqref{end.idem} in Corollary~\ref{cor.end.end.2}, $V_4$ can be made into a truss with the heap operation $[---]_+$ and the multiplication table
$$
\begin{array}{c|cccc}
	\cdot & 0 & a & b  & a+b \\  \hline
 0 	& a & 0 & a & 0   \\ 
 a 	& 0 & a & 0 & a  \\ 
 b 	& a+b & b & a+b & b  \\ 
 a+b 	& b & a+b & b & a+b \\ 
 \end{array}
$$
 
This truss is right braceable but it has neither left nor right absorbers, nor central elements, and thus can serve as an illustration that  
trusses might reach beyond rings: it is not ring-type and there is no associated ring as in Lemma~\ref{lem.ring.cen} either.
\end{ex}

\begin{ex}\label{ex.matrix}
Let $R$ be a ring and let $\mathbf{e}$ be an $n\times n$ idempotent matrix with entries from $R$. Then $R^n$ is a truss with the heap operation induced from the additive group structure of $R^n$ and the multiplication
$$
(r_1,\ldots, r_n)(s_1,\ldots, s_n) = (r_1,\ldots, r_n) + (s_1,\ldots, s_n)\mathbf{e},
$$
by Lemma~\ref{lem.heap.truss} (with $\alpha$ given by right multiplication by $1-\mathbf{e}$).
\end{ex}

\subsection{The truss structures on integers}\label{sec.integer}
The set of integers $\ZZ$ with the usual addition can be viewed as a heap with induced operation,
\begin{equation}\label{heap.+}
[l,m,n]_+ = l-m+n.
\end{equation}
In this section we classify all truss structures on $(\ZZ,[---]_+)$.

Let $M_2(\ZZ)$ be the set of two-by-two matrices with integer entries, and let
$$
\Ii_2(\ZZ) = \{\p \in M_2(\ZZ) \, |\, \p^2 = \p, \, \tr (\p) =1\}.
$$ 
Note that $\Ii_2(\ZZ)$ can be characterised equivalently as the set of all idempotents different from zero and identity. The group of invertible matrices in $M_2(\ZZ)$, $GL_2(\ZZ)$, acts on $\Ii_2(\ZZ)$ by conjugation, $\mathbf{a}\la \mathbf{p} = \mathbf{a}\mathbf{p}\mathbf{a}^{-1}$. 

\begin{thm}\label{thm.integer}
{}~
\begin{zlist}
\item There are two non-commutative truss structures on $(\ZZ,[---]_+)$ with products defined for all $m,n\in \ZZ$,
\begin{equation}\label{mnm}
m\cdot n =m \quad \mbox{or} \quad m\cdot n = n.
\end{equation}
\item 
 Commutative truss structures on $(\ZZ,[---]_+)$  are in one-to-one correspondence with elements of $\Ii_2(\ZZ)$.
\item Let 
$$
 D_\infty = \left\{\begin{pmatrix} 1 & 0 \cr k & \pm 1 \end{pmatrix}\; |\;   k\in \ZZ\right\},
 $$
be an infinite dihedral subgroup of $GL_2(\ZZ)$.
 Isomorphism classes of commutative truss structures on $(\ZZ, [---]_+)$ are in one-to-one correspondence with orbits of the action of $D_\infty$ on $\Ii_2(\ZZ)$.
\end{zlist}
\end{thm}
\begin{proof}
Since the additive group of integers is generated by 1, in view of the truss distributive law, any truss product $\cdot$ on $\ZZ$ is fully determined by $\alpha,\beta,\gamma,\delta \in \ZZ$, defined as
\begin{equation}\label{01}
0\cdot 0 = \alpha, \quad 0\cdot 1 = \beta, \quad 1\cdot 0 = \gamma, \quad 1\cdot 1 = \delta.
\end{equation}
Exploring the truss distributive law we find the following recurrence, for all $n\in \ZZ$,
$$
0\cdot (n+1) = 0\cdot (n -0 +1) = 0\cdot n - 0\cdot 0 +0\cdot 1 = 0\cdot n -\alpha +\beta.
$$
With the initial condition in \eqref{01} this recurrence is easily solved to give
\begin{equation}\label{0n}
0\cdot n = \beta n - \alpha (n-1).
\end{equation}
Replacing $0$ by $1$ in the above recurrence and then swapping the sides we obtain the remaining three relations, for all $m,n\in \ZZ$,
\begin{equation}\label{1n}
1\cdot n = \delta n - \gamma (n-1), \quad m\cdot 0 = \gamma m - \alpha (m-1), \quad m\cdot 1 = \delta m - \beta (m-1).
\end{equation}
Put together equations \eqref{0n} and \eqref{1n} provide one with necessary formula for a product that distributes over $[---]_+$,
\begin{equation}\label{mn}
m\cdot n = \delta mn -\gamma m(n-1) -\beta (m-1)n + \alpha(m-1)(n-1).
\end{equation}

We need to find constraints on the parameters $\alpha,\beta,\gamma$ and $\delta$ arising from the associative law. Rather than studying the general case we first look at special cases to determine necessary conditions. Specifically, the identities $0\cdot (0\cdot 0) = (0\cdot 0)\cdot 0$, $1\cdot (0\cdot 1) = (1\cdot 0)\cdot 1$, $1\cdot (1\cdot 0) = (1\cdot 1)\cdot 0$, $0\cdot (1\cdot 1) = (0\cdot 1)\cdot 1$ yield
$$
\alpha(\beta-\gamma) = 0, \quad (\delta-1)(\beta-\gamma)=0, \quad (\gamma-1)\gamma = (\beta-1)\beta = \alpha(\delta-1).
$$
The non-commutative case corresponds to the choice $\beta\neq \gamma$ and the above equations imply that $\alpha =0$, $\delta =1$ and then either $\beta =0$ and $\gamma =1$ or $\beta=1$ and $\gamma=0$. Inserting these values into \eqref{mn} we obtain formulae \eqref{mnm} which clearly define associative operations. Thus statement (1) follows.

We can now concentrate on the commutative case. Since $\beta=\gamma$, the formula \eqref{mn} can be re-written as
\begin{equation}\label{mn.com}
m\cdot n = amn +b(m+n) +c,
\end{equation}
where $c=\alpha$, $b=\beta-\alpha$, $a=\delta -2\beta +\alpha$. Since the operation $\cdot$ is commutative, the associative law can be re-arranged to
\begin{equation}\label{lmn}
l\cdot (m\cdot n) = n\cdot (m\cdot l),
\end{equation}
and hence it boils down to ensuring the $l$-$n$ symmetry of the formula for the triple product. The $l$-$n$ {\em asymmetric} terms in the left-hand side of \eqref{lmn} are
$$
acl+b^2n+bl,
$$
and thus \eqref{lmn} is equivalent to 
$$
(ac+b-b^2)(l-n)= 0, \qquad \mbox{for all $l,n\in \ZZ$}.
$$
Therefore the product \eqref{mn.com} is associative if and only if
\begin{equation}\label{constraint}
ac = b(b-1).
\end{equation}
One easily checks that the product \eqref{mn.com} distributes over the ternary operation \eqref{heap.+}, and thus we may conclude that all commutative truss structures on $\ZZ$ have product of the form \eqref{lmn} subject to the constraint \eqref{constraint}. The parameters $a,b,c$ can be arranged in a two-by-two  integer matrix
\begin{equation}\label{para}
\p = \begin{pmatrix} b  & a \cr -c & 1-b
\end{pmatrix} \in M_2(\ZZ).
\end{equation}
Using the constraint \eqref{constraint} one easily finds that the characteristic polynomial of $\p$ is $t^2 -t$, and thus $\p$ is an idempotent. Since $\tr(\p) =1$, $\p\in \Ii_2(\ZZ)$, as required.

Conversely,  observe that a general trace one matrix \eqref{para}
 is an idempotent, i.e.\ an element of $\Ii_2(\ZZ)$,  if and only if $b(b-1)=ac$, which is precisely the associativity constraint \eqref{constraint}
This establishes the one-to-one correspondence of assertion (2). 

To prove (3) we first need to identify all automorphisms of $(\ZZ, [---]_+)$. These are in bijective correspondence with the elements of the holomorph of $(\ZZ,+)$ or, equivalently, the elements of the semi-direct product of $\ZZ$ with the automorphism group of $(\ZZ,+)$. The latter is isomorphic to $\ZZ_2$, and thus  $\mathrm{Aut}(\ZZ, [---]_+)$ is isomorphic to the infinite dihedral group. Explicitly,
\begin{equation}\label{z.auto}
\mathrm{Aut}(\ZZ, [---]_+) = \{\varphi_k^\pm \; |\; k\in \ZZ\},\qquad 
\varphi_k^\pm : n\lto k \pm n.
\end{equation}
Let
$$
\p = \begin{pmatrix} b  & a \cr -c & 1-b
\end{pmatrix} ,\; \tilde{\p} = \begin{pmatrix} \tilde{b}  & \tilde{a} \cr -\tilde{c} & 1-\tilde{b}
\end{pmatrix} \in \Ii_2(\ZZ),
$$
and suppose that the corresponding truss products are related by a heap automorphism $\varphi_k^\pm$. Exploring the equality
$$
\varphi_k^\pm(m\cdot n) = \varphi_k^\pm(m)\cdot \varphi_k^\pm(n), \qquad \mbox{for all $m,n\in \ZZ$},
$$
one finds that necessarily 
\begin{equation}\label{z.auto.a}
 \tilde{a} =\pm a, \qquad \tilde{b} = b \mp ak, \qquad \tilde{c} = \pm (c + ak^2) -2bk +k,
\end{equation}
where the upper choice of signs corresponds to $\varphi_k^+$, and the lower one to $\varphi_k^-$. Thus, in the matrix form, 
$$
\begin{pmatrix} \tilde{b}  & \tilde{a} \cr -\tilde{c} & 1-\tilde{b}
\end{pmatrix} = \begin{pmatrix} 1 & 0 \cr k & \pm 1 \end{pmatrix} 
 \begin{pmatrix} b  & a \cr -c & 1-b \end{pmatrix} 
\begin{pmatrix} 1 & 0 \cr \mp k & \pm 1 \end{pmatrix}.
$$
In other words, if the idempotent $\tilde{\p}$ describes the truss structure  isomorphic to that of $\p$, then it is similar to $\p$ with the similarity matrix necessarily in $D_\infty$. Since similarity transformation preserves both traces and the idempotent property, any element of $D_\infty$ corresponds to an isomorphism of trusses.
Therefore, two commutative truss structures on $(\ZZ,[---]_+)$ are isomorphic if and only if the corresponding idempotents belong to the same orbit under the (conjugation) action of
 $D_\infty$ on $\Ii_2(\ZZ)$.
\end{proof}

\begin{rem}\label{rem.integer.0}
Among the truss structures on $\ZZ$ classified in Theorem~\ref{thm.integer} there are three classes which can be defined on any Abelian heap. First, the choice $a=b=0$ yields the constant (fully-absorbing) truss of Lemma~\ref{lem.disc.truss}. Second,  both non-commutative structures are of the type described in Corollary~\ref{cor.heap.truss}. Third, the commutative products given by idempotents with $b=1$ and $a=0$ are of the type described in Corollary~\ref{cor.heap.truss.+} or Corollary~\ref{cor.end.end.2}.
\end{rem}

\begin{cor}\label{cor.integer.class}
Up to isomorphism there are the following commutative products equipping $(\ZZ, [---]_+)$ with different  truss structures (for all $m,n\in \ZZ$):
\begin{zlist}
\item 
\begin{subequations}
\begin{equation}\label{c.1}
m\cdot n = 0,
\end{equation}
\begin{equation}\label{c.2}
m\cdot n =  m +n.
\end{equation}
\end{subequations}
\item For all $a\in \ZZ_+$, 
\begin{subequations}
\begin{equation}\label{a.1}
m\cdot n = amn, 
\end{equation}
\begin{equation}\label{a.2}
m\cdot n = amn + m +n, 
\end{equation}
\end{subequations}
\item For all $a\in \ZZ_+$, $b\in \{2,3,\ldots,a-1\}$ and  $c\in \ZZ_+$ such that $ac = b(b-1)$, 
\begin{equation}
m\cdot n = amn +b(m+n) +c.
\end{equation}
\end{zlist}
\end{cor}
\begin{proof}
We know from (the proof of) Theorem~\ref{thm.integer} that all commutative truss products have the form \eqref{mn.com} where the integers $a,b,c$ are constrained by \eqref{constraint}. Applying the isomorphism $\varphi^+_k$ \eqref{z.auto} with a suitable choice of $k$, in view of \eqref{z.auto.a}, we can always restrict values of $b$ to $\{0,1,\ldots, |a|-1\}$. In that case $b(b-1)\geq 0$, thus both $a$ and $c$ have the same sign or at least one of them is zero. Using $\varphi^-_0$ we can change $a$ to $-a$ and $c$ to $-c$ without affecting $b$. Thus, up to isomorphism, only natural values of $a$ and $c$ need be considered. If $a=0$, then either $b=0$ or $b=1$. In both cases $c$ can be eliminated by a suitable choice of $k$ in \eqref{z.auto.a}. This gives the case (1). If $a\neq 0$ then either $c$ is zero, in which case $b=0$ or $b=1$,  yielding (2), or $c\neq 0$, which gives the structures described in (3).
\end{proof}

\begin{ex}\label{ex.integer}
The number of possible structures of type (3) in Corollary~\ref{cor.integer.class} depends on the value of $a$. If $a=p^l$ for a prime $p$, then in order to satisfy the constraint \eqref{constraint}, $b=p^k$ or $b=p^k+1$ for some $0<k<l$. In the first case, however, $b-1$ is not divisible by $p$, while in the second $b-1$ is not divisible by $p$, hence their product is not divisible by $a=p^l$. There are no structures of type (3) in this case.

If $a=pq$, for $p\neq q$ prime, then by the B\'ezout lemma, there is exactly one pair $(k,l)$, $0<k <q$ and $0<l<p$ such that $kp-lq=1$ in which case $b=kp$ and $c=kl$ solve the constraint \eqref{constraint}, and thus give the product, for all $m,n\in \ZZ$,
$$
m\cdot n = pqmn +kp(m+n) +kl.
$$
Furthermore, there is exactly one pair $(k,l)$, $0<k <q$ and $0<l<p$ such that $-kp+lq=1$ in which case $b=lp$ and $c=kl$ solve the constraint \eqref{constraint}, yielding the product, 
$$
m\cdot n = pqmn +lq(m+n) +kl.
$$
\end{ex}
Solving the unitality and  absorption conditions for truss structures listed in Theorem~\ref{thm.integer}  and Corollary~\ref{cor.integer.class} one obtains
\begin{cor}\label{cor.integer}
\begin{zlist}
\item Any unital truss on $\ZZ$ is isomorphic to the one with the product \eqref{c.2} or \eqref{a.2} in Corollary~\ref{cor.integer.class}.
The identity is 0. 

\item Any ring-type truss on $\ZZ$ is isomorphic to one with the product \eqref{c.1} or \eqref{a.1} in Corollary~\ref{cor.integer.class}.
The absorber is 0.
\end{zlist}
\end{cor}
\begin{proof}
Solving the unitality constraint one obtains that the truss product is necessarily of the form
$$
m\cdot n = amn +(1-au)(m+n) +(au-1)u,
$$
for all $a,u\in \ZZ$. The integer $u$ is the identity. By applying $\varphi^+_{-u}$ this product can be transferred to the form $m\cdot n = amn + m+n$, and the identity comes out as $\varphi^+_{-u}(u) =0$. The existence of $\varphi_0^-$ allows one to restrict  $a$ to be natural. This proves statement (1).

In a similar way, all ring-type trusses are of the form
$$
m\cdot n = amn -az(m+n) +(az+1)z,
$$
for some $z, a\in\ZZ$, and 
 $z$ is the absorber in that case. Applying $\varphi^+_{-z}$ we obtain the product of the type \eqref{a.1}, and $\varphi_0^-$ can be used to make $a$ non-negative.
\end{proof}

\begin{rem}\label{rem.integer}
The arguments of Theorem~\ref{thm.integer} can be applied to any commutative ring $R$. First view $R$ as a heap using its abelian group structure, so that $[r,s,t]_+= r-s+t$. Then, for all $a,b,c\in R$ such that 
\begin{equation}\label{r.constraint}
ac= b^2-b,
\end{equation}
 the product
\begin{equation}\label{r.dot}
r\cdot s =a rs + b(r+s) + c.
\end{equation}
is associative and makes $(R,[---]_+)$ into a commutative truss. We note in passing that the constraint \eqref{r.constraint} on the parameters is not required by the truss distributive law but only by the associative law. If, furthermore, $R$ is unital, then  \eqref{r.constraint} is equivalent to the idempotent property of $\begin{pmatrix} b  & a \cr -c & 1-b
\end{pmatrix} \in M_2(R)$. Also in that case, if $b=1-au$, $c=u(au-1)$ for some $u\in R$, then $(R,[---]_+,\cdot)$ is unital with identity $u$. On the other hand if $b=-az$, $c=z(1+az)$ for an element $z\in R$, then $z$ is the absorber in $(R,[---]_+,\cdot)$.
\end{rem}

\subsection{The mapping trusses}\label{sec.map}
The collection of all mappings from a set to a truss forms a truss.
\begin{lem}\label{lem.map}
Let $(T,[---],\cdot)$ be a truss and let $X$ be a set. The set $T^X$ of all functions $X\to T$ is a truss with the pointwise defined operations, i.e.\ for all $f,g,h\in T^X$,
\begin{subequations}\label{map.tru}
\begin{equation}
[f,g,h] : X\to T, \qquad x\mapsto [f(x),g(x),h(x)],
\end{equation}
\begin{equation}
fg : X\to T, \qquad x\mapsto f(x)g(x).
\end{equation}
\end{subequations}
\end{lem}
\begin{proof}
The assertion is easily checked by a straightforward calculation.
\end{proof}

The sequences of elements in a truss can be truncated by an idempotent element to form a truss.
\begin{lem}\label{lem.poly}
Let $(T,[---],\cdot)$ be a truss and let $e\in T$ be an idempotent element of $(T,\cdot)$. Define the subset of $T^\NN$,
\begin{equation}\label{poly}
T_e[X] = \{ f\in T^\NN\; |\; \exists n\in \NN\; \forall m>n, \; f(m)= e\}.
\end{equation}
Then $T_e[X]$ is a sub-truss of $(T^\NN,[---],\cdot)$.
\end{lem}
\begin{proof}
The statement follows by the idempotent properties of $e$ (both with respect to $[---]$ and $\cdot$).
\end{proof}

\section{Modules}\label{sec.mod}
The search for a representation category of trusses leads in a natural way to the notion of a module. In this section we define modules over trusses and describe their basic properties.

\subsection{Modules: definitions}\label{sec.def.mod}

\begin{defn}\label{def.module} 
Let $(T, [---],\cdot)$ be a truss. A triple $(M,[---],\alpha_M)$ consisting of an Abelian heap $(M,[---])$ and a morphism of trusses
$$
\alpha_M: T \lra E(M),
$$
is called a {\em left module} over $(T, [---],\cdot)$ or, simply, a left $T$-module. 

If $T$ is a unital truss (with identity 1), then a $T$-module $M$ is said to be {\em normalised} provided $\alpha_M(1) = \id_M$.
\end{defn}

\begin{lem}\label{lem.action}
For a truss $(T, [---],\cdot)$ and an Abelian heap $(M,[---])$ the following statements are equivalent.
\begin{zlist}
\item There exists a morphism of trusses $\alpha_M: T \to E(M)$.
\item There exists a mapping 
\begin{equation}\label{action}
\lambda_M: T\times M\longrightarrow M, \qquad (x,m)\mapsto x\la m,
\end{equation}
satisfying the following properties, for all $x,y,z\in T$ and $m,m',m''\in M$,
\begin{rlist}
\item $(xy)\la m = x\la (y\la m)$,
\item $x\la [m,m',m''] = [x\la m, x\la m', x\la m'']$,
\item $[x,y,z]\la m = [x\la m, y\la m, z\la m]$.
\end{rlist}
\end{zlist}
\end{lem}
\begin{proof}
Given a truss morphism $\alpha_M: T \to E(M)$, define 
\begin{equation}\label{pi.action}
\lambda_M:   T\times M\longrightarrow M, \qquad \lambda_M(x,m) = \alpha_M(x)(m).
\end{equation}
Then property (i) for $\lambda_M$ defined by \eqref{pi.action} follows from the fact that $\alpha_M$ is a homomorphism of semigroups and the property (ii) records that, for all $x\in T$, $\alpha_M(x)$ is an endomorphism of heaps (so that it preserves the heap operation on $M$). Finally, property (iii) is a consequence of the fact that $\alpha_M$ is a morphism of heaps, with the heap operation on $E(M)$  defined pointwise.

Conversely, given a mapping $\lambda_M$ that satisfies properties (i)--(iii) in (2), define
\begin{equation}\label{action.pi}
\alpha_M: T \longrightarrow \mathrm{Map} (M, M), \qquad x\longmapsto [m\mapsto \lambda_M(x, m)].
\end{equation}
Then reversing the arguments in the proof of the first implication we can connect the properties (i)--(iii) with the property that $\alpha_M$ defined by \eqref{action.pi} is a morphism of trusses from $T$ to the endomorphism truss $E(M)$.
\end{proof}

\begin{defn}\label{def.action}
A map $\lambda_M$ satisfying properties (i)--(iii) in Lemma~\ref{lem.action}~(2) is called the {\em action} of $T$ on $M$. Often rather than writing $(M,[---], \alpha_M)$ we will write $(M,\lambda_M)$. Typically, we write $x\la m := \lambda_M(x,m)$.
\end{defn}

\begin{rem}\label{rem.right.module}
Symmetrically to a left $T$-module one defines a {\em right $T$-module} as a triple $(M,[---],\alpha^\circ_M)$ in which $\alpha^\circ_M$ is morphism of trusses from the opposite truss $T^\mathrm{op}$ to $E(M)$. The corresponding right action is denoted by $
\varrho_M: M\times T \longrightarrow M$,  $(m,x)\mapsto m\ra x,
$ and it satisfies rules analogous to those in Lemma~\ref{lem.action}.
In general, by the left-right symmetry, whatever is stated for a left $T$-module can equally well be stated for a right $T$-module.
\end{rem}

\begin{defn}\label{def.bimodule}
Given two trusses $(T, [---],\cdot)$ and $(\tilde{T}, [---],\cdot)$ a $(T,\tilde{T})$-bimodule  is a quadruple $(M,[---],\alpha_M,\alpha^\circ_M)$ such that $(M,[---],\alpha_M)$ is a left $T$-module, $(M,[---],\alpha^\circ_M)$ is a right $\tilde{T}$-module, and, for all $x\in T$ and $y\in \tilde{T}$,
\begin{equation}\label{bimodule}
\alpha_M(x)\circ \alpha^\circ_M(y) =  \alpha^\circ_M(y)\circ \alpha_M(x).
\end{equation}
\end{defn}

In terms of the actions, the condition \eqref{bimodule} simply means that for all $x\in T$, $y\in \tilde{T}$ and $m\in M$,
\begin{equation}\label{bimodule.ass}
x\la(m\ra y) = (x\la m)\ra y.
\end{equation}

\begin{ex}\label{ex.module.reg}
A truss $(T, [---],\cdot)$ is its own (left, right, bi-) module with the action(s) given by the multiplication, i.e.,  for all $x,y\in T$,
$$
x\la y = x\ra y = xy.
$$
In a similar way any ideal in $T$ is a $T$-module.
This follows immediately from the associative and distributive laws for trusses or from Lemma~\ref{lem.end.ideal} or Corollary~\ref{cor.end.ideal}.
\end{ex}

\subsection{Modules over $\ZZ$}\label{sec.mod.z}
By fixing an element in an Abelian heap one obtains an Abelian group. Consequently any heap is a module over a ring $\ZZ$; this module structure is unique, if one requests the action to be unital (or normalised). On the other hand the usual ring $\ZZ$ can be understood as a truss (in the classification of Theorem~\ref{thm.integer} the corresponding idempotent matrix is $\begin{pmatrix} 0 & 1\cr 0& 1\end{pmatrix}$). The truss distributive law is more flexible than the ring distributive law. Consequently there is more flexibility for understanding heaps as (normalised) modules over the truss $(\ZZ, [---]_+,\cdot)$. 

\begin{lem}\label{lem.char}
Let $(T, [---],\cdot)$ be a truss and let $e,\iota \in T$ be such that
\begin{equation}\label{e.f}
e^2 =e, \qquad \iota^2 = \iota, \qquad e\iota=\iota e=e.
\end{equation}
Then the function
\begin{equation}\label{char.e.f}
\chi_{e,\iota} : \ZZ \lra T, \qquad n\lto 
\begin{cases} [[\iota,e,\iota,\ldots,e,\iota]]_{n-1}, &n>0,\cr [[e,\iota,e,\ldots,\iota,e]]_{|n|}, & n\leq 0,
\end{cases}
\end{equation}
is a homomorphism of trusses.
\end{lem}
\begin{proof}
In the view of the heap-group correspondence, $\chi_{e,\iota}$ is a group homomorphism from $\ZZ$ to $(T, +_e)$, hence it is a morphism of heaps. The truss distributive laws together with the rules \eqref{e.f} imply that
$$
\chi_{e,\iota} (n)\iota = \iota \chi_{e,\iota} (n) = \chi_{e,\iota} (n), \qquad \chi_{e,\iota} (n)e = e\chi_{e,\iota} (n) = e.
$$
Therefore, for positive $m$,
$$
\begin{aligned}
\chi_{e,\iota} (m)\chi_{e,\iota} (n) &= [[\iota,e,\iota,\ldots,e,\iota]]_{m-1}\chi_{e,\iota} (n) \\
& = [[\iota \chi_{e,\iota} (n),e \chi_{e,\iota} (n),\iota \chi_{e,\iota} (n),\ldots,e\chi_{e,\iota} (n),\iota\chi_{e,\iota} (n)]]_{m-1}\\
& = [[\chi_{e,\iota} (n),e ,\chi_{e,\iota} (n),\ldots,e,\chi_{e,\iota} (n)]]_{m-1} = \chi_{e,\iota}(mn),
\end{aligned}
$$
where the Mal'cev reduction of the consecutive $e$ is used in the case of a non-positive $n$. The case of non-negative $m$ is dealt with in a similar way.
\end{proof}

\begin{cor}\label{cor.char}
Let $(H,[---])$ be an Abelian heap and let $\eps,\iota$ be idempotent endomorphisms of $(H,[---])$ such that $\eps\circ\iota = \iota\circ\eps = \eps$. Then $H$ is a module over the truss $(\ZZ,[---]_+,\cdot)$, where $\cdot$ is the usual multiplication of integers, with the action, for all $n\in \ZZ$ and $x\in H$,
\begin{equation}\label{z.act}
n\la x = \begin{cases} [[\iota(x),\eps(x),\iota(x),\ldots,\eps(x),\iota(x)]]_{n-1}, &n>0,\cr [[\eps(x),\iota(x),\eps(x),\ldots,\iota(x),\eps(x)]]_{|n|}, & n\leq 0.
\end{cases}
\end{equation}
\end{cor}
\begin{proof}
In view of Lemma~\ref{lem.char}, the maps $\eps$, $\iota$ induce a truss homomorphism from $\ZZ$ to the endomorphism truss of $H$. The resulting action comes out as in the statement of the corollary.
\end{proof}

\begin{rem}\label{rem.char}
Choosing $\iota = \id$ in Corollary~\ref{cor.char} one can connect normalised modules over $\ZZ$ with idempotents in $E(H)$. The latter have been discussed in the proof of Corollary~\ref{cor.end.end.2}, and identified with pairs consisting of idempotents $\alpha$ in the endomorphism ring of any associated group $(H,+_e)$ and elements $a\in H$ such that $\alpha(a) =e$. 

Making suitable choices, one finds, for example that $\ZZ$ acts on itself by $m\la n = n$ or $m\la n = mn-(m-1)a$, for all $m,n,a\in \ZZ$. 
\end{rem}

\subsection{Products of modules and function modules}
The following lemmas are established by straightforward calculations:
\begin{lem}\label{lem.prod.mod}
Let $(T, [---],\cdot)$ be a truss and $(M,[---], \alpha_M)$ and $(N,[---],\alpha_N)$ two left $T$-modules. Then $M\times N$ is a $T$-module with the product heap and module structures, i.e.\ 
\begin{blist}
\item with the heap operation defined by
$$
\left[(m_1,n_1), (m_2,n_2), (m_3,n_3)\right] = \left([m_1,m_2,m_3],[n_1,n_2,n_3]\right),
$$
for all $m_1,m_2,m_3\in M$, $n_1,n_2,n_3 \in N$;
\item the module structure map
$$
\alpha_{M\times N}: T \lra E(M\times N), \qquad x\lto \left(\alpha_M(x), \alpha_N(x)\right),
$$
i.e.\ for all $x\in T$, $m\in M$ and $n\in N$,
$$
x\la (m,n) = (x\la m, x\la n).
$$
\end{blist}
\end{lem}
The construction of Lemma~\ref{lem.prod.mod} can be iterated to obtain a coproduct of modules. In a similar way,
\begin{lem}\label{lem.func.mod}
Let $(T, [---],\cdot)$ be a truss and let $(M,[---], \alpha_M)$ be a left $T$-module. For any set $X$, the heap $M^X$ of functions from $X$ to $M$ is a module with a pointwise defined action, for all $t\in T$, $x\in X$ and $f\in M^X$,
$$
(t\la f) (x) = t\la f(x),
$$
i.e.\ $\alpha_{M^X}(t) = \mathrm{Map}(X,\alpha_M(t))$.
\end{lem}

\subsection{Morphisms of modules}
\begin{defn}\label{def.mor.mod}
Let $(M,\lambda_M)$ and $(N,\lambda_N)$ be left modules over a truss $(T,[---],\cdot)$. A {\rm morphism} from $M$ to $N$ is a homomorphism of heaps
$\varphi: M\to N$ rendering commutative the following diagram
\begin{equation}\label{mor.mod}
\xymatrix{ T\times M\ar[rr]^-{\id \times \varphi}\ar[d]_{\lambda_M} && T\times N\ar[d]^{\lambda_N}\\
M \ar[rr]^-\varphi && N.}
\end{equation}
The set of all morphisms from $M$ to $N$ is denoted by $\lhom T MN$.
\end{defn}

In a symmetric way morphisms of right $T$-modules are defined, and their set denoted by $\rhom T MN$. If $M$ and $N$ are $(T,S)$-bimodules, then their morphisms are defined as
$$
\lrhom T S M N := \lhom T MN \cap \rhom SMN.
$$
The categories of left $T$-, right $T$-, $(T,S)$-bi-modules are denoted by ${}_T\mod$, $\mod_T$ and ${}_T\mod_S$, respectively.

\begin{lem}\label{lem.mor.mod}
The set $\lhom T MN$ is a heap with the pointwise heap operation.
\end{lem}
\begin{proof}
Suffices it to check whether, for all $\varphi_1, \varphi_2,\varphi_3\in \lhom T MN$, the map
$$
[\varphi_1, \varphi_2,\varphi_3]: M\lra N, \qquad m \longmapsto  [\varphi_1(m), \varphi_2(m),\varphi_3(m)],
$$
is a morphism of modules. Note that the commutativity of diagram \eqref{mor.mod} for $\varphi_i$, $i=1,2,3$, means that for all $x\in T$ and $m\in M$, $\varphi_i(a\la m)  = a\la \varphi_i(m)$. That this property holds also for  the result of the ternary operation $[\varphi_1, \varphi_2,\varphi_3]$ follows from the distributive law from Lemma~\ref{lem.action}~(2)(ii). 
\end{proof}

\begin{prop}\label{prop.mor.mod} 
Let $T$ and $S$ be trusses and let $M\in {}_T\mod_S$ and $N\in  {}_T\mod$.
\begin{zlist}
\item The heap $\lhom T MN$ is a left $S$-module with the action given by
$$
(x \la \varphi ) (m) = \varphi(m \ra x), \qquad \mbox{for all $\varphi \in \lhom T MN$, $m\in M$, $x\in S$}.
$$
\item The heap $\lhom T NM$ is a right $S$-module with the action given by
$$
 (\varphi \ra x) (n) = \varphi(n )\ra x, \qquad \mbox{for all $\varphi \in \lhom T NM$, $n\in N$, $x\in S$}.
$$
\end{zlist}
\end{prop}
\begin{proof}
(1) That $x\la \varphi$ is a morphism of heaps follows by the distributive law for right actions. The preservation of the action, i.e.\ the commutativity of the diagram \eqref{mor.mod} for $x\la \varphi$ is a  consequence of the bimodule condition \eqref{bimodule.ass}. Therefore, the formula in statement (1) gives a mapping
$$
\lambda: S\times \lhom T MN \lra \lhom T MN, \qquad (x,\varphi) \lto x\la \varphi.
$$
The associativity of the induced left action $\lambda$, Lemma~\ref{lem.action}~(2)(i), follows by the associative law for a right action. The distributive laws for actions, properties (ii) and (iii) in assertion (2) of Lemma~\ref{lem.action} follow by the corresponding properties of a right action combined with the pointwise definition of the heap operation on $\lhom T MN$.

(2) Combining the fact that $\varphi$ is a morphism of heaps with the distributive law of right actions one finds that $\varphi\ra x$ is a morphism of heaps. As was the case in the proof of assertion (1), the bimodule associative law \eqref{bimodule.ass} implies that $\varphi\ra x$ preserves the actions, i.e.\ makes the right-action version of the diagram \eqref{mor.mod} commute. Consequently, the formula in statement (2) gives a mapping
$$
\varrho: \lhom T NM\times S\lra \lhom T NM, \qquad (\varphi, x)\lto \varphi \ra x.
$$
That $\varrho$ is a right action  follows by the fact that  $\ra: M\times S\to M$ is such an action.
\end{proof}

\begin{rem}\label{rem.hom.functor}
The constructions in Proposition~\ref{prop.mor.mod} yield functors.
\begin{zlist}
\item The covariant Hom-functor,
$$
\begin{aligned}
 \lhom T M - : {}_T\mod \lra {}_S\mod, \quad &  N\lto \lhom TMN  \\
 \forall f\in \lhom TLN, \quad  \lhom T M f :  \lhom TML & \lra  \lhom TMN,   \\
 \varphi & \lto f\circ\varphi.
\end{aligned}
$$
\item The contravariant Hom-functor,
$$
\begin{aligned}
 \lhom T - M : {}_T\mod \lra \mod_S, \quad &  N\lto \lhom TNT  \\
 \forall f\in \lhom TLN, \quad  \lhom TfM :  \lhom TNM & \lra  \lhom TLM,   \\
 \varphi & \lto \varphi \circ f.
\end{aligned}
$$
\end{zlist}
\end{rem}

\subsection{Module structures on heaps and paragons}\label{sec.module.paragon}
In this section we construct a functor from the category of groups or, equivalently, based heaps to that of modules, and we also show that every paragon is a module.

\begin{prop}\label{prop.mod.heap}
Let $(T, [---],\cdot)$ be a truss, and let $(H,[---])$ and $(K,[---])$ be  Abelian heaps. 
\begin{zlist}
\item For all $e\in H$, the map
$$
\alpha_e : T\lra E(H), \qquad x\lto [h\lto e],
$$
defines a $T$-module structure on $H$.
\item For all $e,\tilde{e} \in H$, the modules $(H, [---], \alpha_e)$, $(H, [---],\alpha_{\tilde{e}})$ are mutually isomorphic.
\item For all morphisms of heaps $\varphi:H\to K$ and for all elements $e\in H, f\in K$, the heap homomorphism
$$
\varphi_e^f = \tau_{\varphi(e)}^f\circ\varphi: H\lra K, \qquad h\lto [\varphi(h),\varphi(e), f],
$$
is a morphism of modules from $(H, [---], \alpha_e)$ to $(K, [---], \alpha_f)$.
\end{zlist}
\end{prop}
\begin{proof}
(1) Since $[---]$ is an idempotent operation, for all $x\in T$, $\alpha_e(x)$ are idempotents in the endomorphism truss $E(H)$, and $\alpha_e: T\to E(H)$ is a morphisms of heaps. 
Finally,  for all $x,y\in T$ and $h\in H$, 
$$
\alpha_e(xy)(h) = e = \alpha_e(x)(e) = \alpha_e(x)(\alpha_e(y)(h)) = (\alpha_e(x)\circ \alpha_e(y))(h),
$$
hence  $\alpha_e$ preserves binary operations, and thus it is a morphism of trusses.

(2) 
The actions corresponding to $\alpha_e$ and $\alpha_{\tilde{e}}$ come out as 
$
x\la h = e$ and  $x\,\widetilde{\la}\, h = \tilde{e}$, respectively, for all $x\in T$, $h\in H$. The swap automorphism  $\tau_e^{\tilde{e}}$ (see \eqref{iso.ef})  preserves these actions, since
$
\tau_e^{\tilde{e}}(x\la h)= \tau_e^{\tilde{e}}(e) = \tilde{e} = x\,\widetilde{\la}\, \tau_e^{\tilde{e}}(h),
$
and thus it is an isomorphism of modules.

(3) Since both $\varphi$ and $\tau_e^f$ are heap homomorphisms, so is $\varphi^f_e$, as stated.
The Mal'cev identity together with the definition of structure maps $\alpha_e$ and $\alpha_f$ imply that $\varphi^f_e$ preserves actions. 
\end{proof}

The situation described in Proposition~\ref{prop.mod.heap} parallels that of modules over rings: every Abelian group can be made into a (trivial) module over any ring, by the action that sends all pairs of elements (from the ring and the group) to the neutral element (zero) of the group. In contrast to the case of modules over rings, where for an Abelian group  there is only one action of this type, for modules over trusses there are as many actions as there are elements of the heap, albeit every choice leading to an isomorphic module. As was the case for heaps, the category of modules has a terminal object: the singleton set, but no initial objects. Global points of a module over a truss coincide with its elements as a set.  The contents of Proposition~\ref{prop.mod.heap} can be summarised as
\begin{cor}\label{cor.mod.heap}
For any truss $(T,[---],\cdot)$ there is a functor 
$$
\begin{aligned}
 (\{*\}\!\downarrow\! \ahrd ) \lra  {}_T\mod\, , \quad & (H,[---],e)\lto (H,[---],\alpha_e)\\
 & \left[\xymatrix{(H,[---],e)\ar[r]^-\varphi & (K,[---],f)}\right]\lto \varphi^f_e.
 \end{aligned}
$$ 
\end{cor}
Note that rather than taking the co-slice category of Abelian heaps as the domain of the functor in Corollary~\ref{cor.mod.heap} one can take the category of Abelian groups.

\begin{prop}\label{prop.module.paragon}
Let $P$ be a left paragon in $(T, [---],\cdot)$. 
\begin{zlist}
\item For any $e\in P$, $P$ is a left $T$-module by
\begin{equation}\label{module.paragon}
\alpha_e : T\lra E(P),\qquad x \lto [p\to \lambda^e(x,p)],
\end{equation}
where $\lambda^e$ is defined by \eqref{act}.
\item For all $e,\tilde{e}\in P$, the modules $(P,\alpha_e)$ and $(P,\alpha_{\tilde{e}})$ are mutually isomorphic.
\end{zlist}
\end{prop}
\begin{proof}
(1) By the definition of a paragon, the value of $\alpha_e$ is in the set of endomaps of $P$. Proposition~\ref{prop.act} implies that, in fact, for all $x\in T$, $\alpha_e(x)$ is in endomorphisms of $(P,[---])$ (see equation \eqref{act.2}), so it is well defined. Note that, in terms of $\lambda^e$, the corresponding action $\la_e$ is
$
x\la_e p = \lambda^e(x,p),
$
and since $\alpha_e(x)$ is an endomorphism of heaps, condition (2)(ii) in Lemma~\ref{lem.action} is satisfied.
The associativity of action (condition (2)(i) in Lemma~\ref{lem.action}) follows by \eqref{act.1}. Finally, for all $x,y,z\in T$ and $p\in P$,
$$
\begin{aligned}
{}[x,y,z]\la_e p &= \left[e, [x,y,z]e, [x,y,z]p\right] = \left[[e,e,e], [xe,ye,ze], [xp,yp,zp]\right] \\
&=  \left[[e,xe,xp], [e,ye,yp], [e,ze,zp]\right] = \left[x\la_e p,y\la_e p,z\la_e p\right],
\end{aligned}
$$
where we used the distributivity, the idempotent property of $[---]$ and Lemma~\ref{lem.equal}. This proves that the condition (2)(iii) in Lemma~\ref{lem.action} is satisfied, and hence $P$ is a left $T$-module with structure map \eqref{module.paragon}.

(2) We will show that the swap automorphism $\tau_e^{\tilde{e}}$ of the heap $P$ (see  \eqref{iso.ef}) is an isomorphism of $T$-modules. 
For all $x\in T$ and $p\in P$,
$$
\begin{aligned}
\tau_e^{\tilde{e}}(x\la_e p) &= [\tilde{e}, e, [e,xe,xp]] =  [\tilde{e},xe,xp]
 =[\tilde{e}, x\tilde{e}, [x\tilde{e},xe,xp]] =x\la_{\tilde{e}} \tau_e^{\tilde{e}}(p),
\end{aligned}
$$
by the associativity, Mal'cev identities and the (left) distributive law of trusses. Hence $\tau_e^{\tilde{e}}$ is an isomorphism of modules, as required.
\end{proof}

Since any ideal and any truss are paragons, Proposition~\ref{prop.module.paragon} equips ideals and trusses with module structures different from those discussed in Example~\ref{ex.module.reg}.

\subsection{Submodules and quotient modules}

\begin{defn}\label{def.submodule}
Let $(M,[---],\alpha_M)$ be a (left) module over a truss $(T,[---],\cdot)$. A sub-heap $N$ of $(M,[---])$ is called a {\em submodule}, if for all $x\in T$,
$\alpha_M(x)(N) \subseteq N$. 
\end{defn}

In other words a submodule of $(M,[---],\alpha_M)$ is a subset that is closed both under the heap operation and the action $\lambda_M$. Similarly to ideals, it is clear that a non-empty intersection of submodules is a submodule.

\begin{lem}\label{lem.submod.rel}
Let $N$ be a submodule of a $T$-module $(M,[---],\alpha_M)$. The sub-heap relation $\sim_N$ is a congruence in $(M,[---],\alpha_M)$.
\end{lem}
\begin{proof}
If $m\sim_N m'$, then there exists $n\in N$ such that $[m,m',n]\in N$. Since a submodule is closed under the action, $x\la n\in N$, for all $x\in T$, and hence
$$
[x\la m, x\la m', x\la n] = x\la [m,m',n] \in N,
$$ 
i.e.\ $x\la m\sim_N x\la m'$ as required.
\end{proof}

\begin{cor}\label{cor.mod.quotient}
For any submodule $N$ of a $T$-module $(M,[---],\alpha_M)$, the quotient heap $M/N$ is a $T$-module with the induced action $x\la \overline{m} = \overline{x\la m}$, for all $x\in T$ and $m\in M$.
\end{cor}

\begin{defn}\label{def.submod.gen}
Let $X$ be a non-empty subset of a $T$-module $M$. The {\em submodule generated by $X$} is defined as the intersection of all submodules of $M$ containing $X$, and is denoted by $TX$. In case $X= \{e\}$ is a singleton set, we write $Te$ for the module generated by $X$ and call it a {\em cyclic module}.
\end{defn}

Similarly to the description of principal ideals in Section~\ref{sec.ideal}, 
every element $m$ of $Te$ can be written as 
\begin{equation}\label{cyclic}
m=[[m_1, m_2, \ldots, m_{2n+1}]]_n,
\end{equation}
where the double-bracket means the reduction through any placement of the heap operation $[---]$, see \eqref{multi}, and $m_i = e$ or $m_i = t_i\la e$, for some $t_i\in T$.

If $M$ is a normalised module over a unital truss $T$, then $Te = \{t\la e \; |\; t\in T\}$. 

\subsection{Absorption}
Similarly to ring-type trusses if a module has an element which behaves in a way reminiscent of that of the zero in a module over a ring, then a group structure can be chosen over which the action will distribute.
\begin{defn}\label{def.mod.absorber}
An element $e$ of a left $T$-module $(M,[---],\alpha_M)$ is called an {\em absorber}, if, for all $x\in T$, $\alpha_M(x)(e) =e$, i.e.\ $x\la e= e$.
\end{defn}

\begin{ex}\label{ex.mod.absorber}
In the module $(M,[---], \alpha_e)$ of Proposition~\ref{prop.mod.heap}, $e$ is an absorber.  In view of equation \eqref{act.absorb}, in Remark~\ref{rem.act.skew} an element $e\in P$ is an absorber of the action $\lambda^e$ of a truss $T$ on its paragon $P$;  see Proposition~\ref{prop.module.paragon}.
\end{ex}

\begin{lem}\label{lem.mor.absorber}
Let $(M,[---],\alpha_M)$ and $(N,[---],\alpha_N)$ be modules over a truss $T$. A constant morphism of heaps 
$$
\varphi:  N \lra M, \qquad n\lto e,
$$
is a morphism of modules if and only if $e$ is an absorber in $M$.
\end{lem}
\begin{proof}
If $\varphi$ is a morphism of modules, then for all $x\in T$ and any $n\in N$,
$$
x\la e = x\la \varphi(n) = \varphi(x\la n) = e,
$$
i.e.\ $e$ is an absorber. The converse follows by rearranging the order of equalities in the preceding calculation.
\end{proof}

\begin{lem}\label{lem.mod.absorber}
 If $e$ is an absorber in a left $T$-module $(M,[---],\alpha_M)$, then the action of $T$ on $M$ distributes over the binary group operation $+_e = [-,e,-]$.
\end{lem}
\begin{proof}
The statement follows immediately from the absorber property and the distributive law in statement (2)(ii) of Lemma~\ref{lem.action}.
\end{proof}

\begin{prop}\label{prop.mod.absorber}
Let $(M,[---],\alpha_M)$ be a left module over a truss $(T,[---],\cdot)$. Then, for all $e\in M$ there exist a module $M_e$ and a module homomorphism $\varphi_e: M\to M_e$ such that $\varphi_e(e)$ is an absorber in $M_e$, and which have the following  universal property. For all $T$-modules and module morphisms $\psi: M\to N$ that map $e$ into an absorber in $N$ there exists a unique filler (in the category of $T$-modules) of the following diagram
$$
\xymatrix{M \ar[rr]^-{\varphi_e} \ar[dr]_\psi&&M_e \ar@{-->}[dl]^-{\psi_e} \\& N. &}
$$
The module $M_e$ is unique up to isomorphism.
\end{prop}
\begin{proof}
The proof follows that of Proposition~\ref{prop.uni.ringable}.
Let $M_e = M/Te$, the quotient of $M$ by the cyclic submodule generated  by $e$, and let $\varphi_e: M\to M_e$ be the canonical surjection, $m\mapsto \bar{m}$. Then $\varphi_e(e) = \bar{e}$ is an absorber, since for all $x\in T$,
$$
x\la \varphi_e(e) = x\la \bar{e} = \overline{x\la e} = \overline{e},
$$
since $x\la e \in Te$ and hence, by Proposition~\ref{prop.rel}, $\overline{x\la e} = Te = \overline{e}$.

Since $\psi(e)$ is an absorber and $\psi$ is a module morphism, for all $x\in T$,
\begin{equation}\label{psi.absorb}
\psi(x\la e) = x\la \psi(e) = \psi(e).
\end{equation}
If $n\in Te$,  then its presentation \eqref{cyclic} together with \eqref{psi.absorb}, the fact that $\psi$ is a heap morphism and that $[---]$ (and hence any of $[[-\ldots-]]$) is an idempotent operation imply that $\psi(n) = \psi(e)$. Hence, if $m\sim_{Te} m'$, i.e.\ there exist $n,n'\in Te$ such that
$
[m,m',n] = n',
$
then
$$
[\psi(m),\psi(m'), \psi(e)] = [\psi(m),\psi(m'), \psi(n)] = \psi\left([m,m',n]\right) = \psi(n') = \psi(e).
$$
Therefore, $\psi(m)= \psi(m')$ by Lemma~\ref{lem.equal}, and thus there is the function 
$
\psi_e: M_e \to N$,  $\bar{m}\mapsto \psi(m).
$
 Since $\psi$ is a morphism of $T$-modules, so is $\psi_e$. By construction, $\psi_e\circ \varphi_e = \psi$. The uniqueness of both $\psi_e$ and $M_e$ is clear (the latter by the virtue of the universal property by which $M_e$ is defined).
\end{proof}

\subsection{Induced actions}\label{sec.iterated}
Any module over a truss induces a family of isomorphic modules with absorbers.

\begin{prop}\label{prop.iterated}
Let $(M,\lambda_M)$ be a left $(T,[---],\cdot)$-module. Then, for all $e\in M$, $M$ is a $T$-module with the induced action
\begin{equation}\label{induced}
\begin{aligned}
\lambda_M^e: T\times M\lra M, \qquad (x,m)\lto x \la^e m &= [e,\lambda_M(x,e),\lambda_M(x,m)] \\
&= [\lambda_M(x,m), \lambda_M(x,e), e].
\end{aligned}
\end{equation}
For different choices of $e$ induced modules are isomorphic. The induced module $(M,\lambda_M^e)$ has an absorber $e$.
\end{prop}
\begin{proof}
We write $x\la m =\lambda_M(x,m)$, so that $x\la^e m = [x\la m,x\la e,e]$. First we will check that $\lambda_M^e$ is an associative and distributive action. For all $x,y\in T$ and $m\in M$,
$$
\begin{aligned}
x\la^e \left(y\la^e m\right) &= 
\left[x\la [y\la m,y\la e,e], x\la e, e\right]
= \left[ x\la \left(y\la m\right),x\la \left(y\la e\right), e\right]\\
&= \left[ \left(xy\right)\la m,\left(xy\right)\la e, e\right] =\left(xy\right)\la^e m,
\end{aligned}
$$
by the associativity and distributivity of the action $\la$ and by heap axioms. Since the action $\la$ is distributive, we can compute, for all $x,y,z\in T$ and $m\in M$,
$$
\begin{aligned}
{}[x,y,z]\la^e m &= \left[[x,y,z]\la m, [x,y,z]\la e, e\right]\\
&= \left[[x\la m,y\la m,z\la m], [ x\la e, y\la e ,z\la e], [e,e,e]\right]\\
&= \left[[x\la m,x\la e,e], [y\la m,y\la e,e], [y\la m,y\la e,e]\right] \\
& = \left[x\la^e m,y\la^e m,z\la^e m\right],
\end{aligned}
$$
where the idempotency of the heap operation and statement (4) of  Lemma~\ref{lem.equal} have been used too. Using exactly the same arguments one shows that $\la^e$ satisfies the left distributive law too.

Given two elements $e,f\in M$, consider the swap automorphism $\tau_e^f$ of $M$, \eqref{iso.ef}.
For all $x\in T$, $m\in M$,
$$
\tau_e^f\left(x\la^e m\right) = \left[\left[x\la m, x\la e,e\right], e,f\right] = \left[x\la m, x\la e,f\right],
$$
by the associativity of $[---]$ and Mal'cev identities. On the other hand
$$
\begin{aligned}
x\la^f\tau_e^f(m) &= \left[ x\la\left[m,e,f\right], x\la f,f\right]= \left[x\la m, x\la e,f\right],
\end{aligned}
$$
by the left distributivity of $\la$ and the heap axioms. Therefore, $\tau_e^f$ is the required isomorphism of modules.

That $e$ is an absorber in $(M,\lambda_M^e)$ follows immediately from the Mal'cev identities and the definition of $\lambda^e_M$ in  \eqref{induced}.
\end{proof}

One might wonder whether using this induction procedure it is possible to generate a sequence of non-isomorphic modules. The answer to this question is negative.

\begin{lem}\label{lem.iterated}
Let $(M,\lambda_M)$ be a left module over a truss $(T,[---],\cdot)$. Then, for all $e,f\in M$, $(M,\lambda_M^f) = (M,\lambda_M^{e\,f})$ (i.e.\  repetitions of the induction procedure described in Proposition~\ref{prop.iterated} stabilise after the first step).
\end{lem}
\begin{proof}
This is proven by a simple calculation, which uses the heap axioms as well as the derived associativity property in Lemma~\ref{lem.equal}(2). Explicitly, for all $x\in T$ and $m\in M$,
$$
\begin{aligned}
x{\la^e}^f m &= \left[x\la^e m, x\la^e f, f\right] = \left[\left[x\la m, x\la e ,e\right], \left[x\la f, x\la e,e\right], f\right]\\
&= \left[\left[\left[x\la m, x\la e ,e\right], e, x\la e\right] ,x\la  f, f\right]
 = \left[x\la m, x\la f , f\right] = x \la^f m,
\end{aligned}
$$
as required.
\end{proof}

\subsection{Induced submodules}\label{sec.iter.submod}
While the quotient of a module by a submodule is  necessarily a module with an absorber, more general quotients can be obtained by using submodules of the induced module.
\begin{defn}\label{def.iter.submod}
Let $(M,\lambda_M)$ be a  left module over a truss $(T,[---],\cdot)$. A sub-heap $N$ of $M$ is called an {\em induced submodule} if there exists $e\in N$ for which $N$ is a submodule of $(M,\lambda^e_M)$. 
\end{defn}

The calculation of the proof of Proposition~\ref{prop.iterated} immediately confirms that any submodule of $(M,\lambda_M)$ is an induced submodule. 

\begin{lem}\label{lem.iter.submod}
If $N$ is an induced submodule of $(M,[---],\lambda_M)$, then $N$ is a submodule of $(M,\lambda_M^e)$ for all $e\in N$.
\end{lem}
\begin{proof}
Since $N$ is a sub-heap of $(M,[---])$, for all $e,f\in N$, the  swap automorphism $\tau_e^f$ of $M$   restricts to an automorphism of $N$. Therefore, if the action $\lambda^e_M$ restricts to $N$, so does the action $\lambda^f_M$, as it is given by the formula
$$
x\la^f n = \tau_e^f\left(x\la^e\varphi^{-1}(n)\right).
$$
The assertion follows from this.
\end{proof}

The role that induced submodules play in category of modules is revealed by the following proposition
\begin{prop}\label{prop.ker.iter}
Let $(M,[---],\lambda_M)$ be a left module over a truss $(T,[---],\cdot)$.
\begin{zlist}
\item The kernel of a morphism of $T$-modules is an induced submodule of the domain.
\item If $N$ is a sub-heap of $M$, then the quotient $M/N$ has a $T$-module structure such that the canonical epimorphism $\pi_N: M\lra M/N$ is a module morphism if and only if $N$ is an induced submodule of $M$.
\end{zlist}
\end{prop}
\begin{proof}
(1) Take a morphism of modules $\varphi: M\to \tilde{M}$ and take any $e\in \im \varphi$. If $\varphi(m)=e=\varphi(n)$, then, for all $x\in T$,
$$
\begin{aligned}
\varphi\left(x\la^n m\right) &= \varphi\left(\left[x\la m, x\la n,n\right]\right)
= \left[x\la \varphi\left( m\right), x\la \varphi\left(n\right),\varphi\left(n\right)\right]
=e,
\end{aligned}
$$
by the definition of a module homomorphism and one of the Mal'cev identities. Therefore $\ker \varphi$ is closed under induced actions, and so it is an induced submodule of $M$.

(2) If $M/N$ is a module and $\pi_M$ is a module homomorphism, then, since $N=\ker_N(\pi_N)$, $N$ is an induced submodule by statement (1). In the converse direction, assume that $N$ is an induced submodule of $M$. We need to check whether the sub-heap relation preserves the action, i.e.\ that for all $x\in T$ and $m,m'\in M$, if $m\sim_N m'$, then $x\la m\sim_N x\la m'$. By definition $m\sim_N m'$ if and only if there exists $n\in N$ such that $[m,m',n]\in N$. Since $N$ is an induced submodule 
$$
\begin{aligned}
N\ni  x\la^n[m,m',n] &= \left[x\la [m,m',n], x\la n, n\right] 
=\left[ x\la m,x\la m', n\right],
\end{aligned}
$$
by the distributive law of actions on Mal'cev identities, so $x\la m\sim_N x\la m'$, as required. Therefore, $M/N$ is a module such that $\pi_N$ is a module homomorphism, as required.
\end{proof}

 \numberwithin{equation}{section}
  \numberwithin{thm}{section}

\section*{Acknowledgements}
I would like to thank Bernard Rybo\l owicz for interesting comments and discussions.
This research is partially supported by the Polish National Science Centre grant 2016/21/B/ST1/02438.

\end{document}